\documentclass[12pt]{amsart}
\usepackage{amsmath}
\usepackage{amssymb}
\usepackage{mathrsfs}
\usepackage[top=3cm,bottom=3.5cm,right=2.5cm,left=2.5cm]{geometry}
\usepackage{ulem}
\newcommand{\Z}{\mathbb{Z}}
\newcommand{\Q}{\mathbb{Q}}
\newcommand{\N}{\mathbb{N}}
\newcommand{\R}{\mathbb{R}}

\newcommand{\SL}{{\rm SL}}
\newcommand{\dist}{{\rm dist}}
\newcommand{\diam}{{\rm diam}}

\newcommand{\Int}{{\rm Int}}
\newcommand{\Bad}{\boldsymbol{{\rm Bad}}}
\usepackage{amsthm}
\newtheorem{dfn}{Definition}[section]
\newtheorem{prop}{Proposition}[section]
\newtheorem{conj}[prop]{Conjecture}
\newtheorem{thm}[prop]{Theorem}
\newtheorem{cor}[prop]{Corollary}
\newtheorem{lem}[prop]{Lemma}
\theoremstyle{definition}

\usepackage{graphicx}%forarXiv

\makeatletter
\@addtoreset{equation}{section}

\makeatother

\title{On a lower bound of the number of integers in Littlewood's conjecture}
\author{Shunsuke Usuki}
\date{}
\begin{document}
	\maketitle
	
	\begin{abstract}
We show that, for any $0<\gamma<1/2$, any $(\alpha,\beta)\in\R^2$ except on a set with Hausdorff dimension about $\sqrt{\gamma}$, any small $0<\varepsilon<1$ and any large $N\in\N$, the number of integers $n\in[1,N]$ such that
$n\langle n\alpha\rangle\langle n\beta\rangle<\varepsilon$ is greater than $\gamma\varepsilon\log N$ up to a uniform constant. 
This can be seen as a quantitative result on the fact that the exceptional set to Littlewood's conjecture has Hausdorff dimension zero, obtained by M. Einsiedler, A. Katok and E. Lindenstrauss in 2000's.
For the proof, we study the behavior of the empirical measures with respect to the diagonal action on $\SL(3,\R)/\SL(3,\Z)$ and show that we can obtain a quantitative result on Littlewood's conjecture for $(\alpha,\beta)$ if the corresponding empirical measures are well-behaved. We also estimate Hausdorff dimension of the exceptional set to be small.
	\end{abstract}
	
	\section{Introduction}\label{Introduction}
	
	\subsection{Littlewood's conjecture and our main results}\label{intro_littlewood}
	
	{\it Littlewood's conjecture} is the following famous and long-standing problem in simultaneous Diophantine approximation.
	
	\begin{conj}[Littlewood (c. 1930)]\label{LittlewoodConjecture}
		For every $(\alpha,\beta)\in\R^2$,
		\begin{equation}\label{Littlewood_liminf}
			\liminf_{n\to\infty}n\langle n\alpha\rangle\langle n\beta\rangle=0,
		\end{equation}
		where $\langle x\rangle=\min_{n\in\Z}|x-n|$ for $x\in\R$.
	\end{conj}
	
	It is easily seen that a.e. $(\alpha,\beta)\in\R^2$ satisfies the equation (\ref{Littlewood_liminf}), since a.e. $\alpha\in\R$ satisfies $\liminf_{n\to\infty}n\langle n\alpha\rangle=0$. Hence (\ref{Littlewood_liminf}) is nontrivial when $\alpha$ and $\beta$ are badly approximable, that is, they satisfy $\liminf_{n\to\infty}n\langle n\alpha\rangle>0$. However, if we write $\Bad$ for the set of all badly approximable numbers, then Hausdorff dimension of $\Bad$ is $1$ and there are many $(\alpha,\beta)\in\R^2$ for which (\ref{Littlewood_liminf}) is nontrivial.
	
	There are remarkable works toward Conjecture \ref{LittlewoodConjecture}. J.W.S. Cassels and H.P.F. Swinnerton-Dyer made the first development in \cite{CSD55} and showed that the equation (\ref{Littlewood_liminf}) holds if $\alpha$ and $\beta$ are in the same cubic number field. For a long time then, there had been no remarkable development about Conjecture \ref{LittlewoodConjecture}. However, A.D. Pollington and S. Velani showed in \cite{PV00} that, for each $\alpha\in\Bad$, there exists a subset $\boldsymbol{G}(\alpha)$ of $\Bad$ such that $\dim_H\boldsymbol{G}(\alpha)=1$ and, for every $\beta\in\boldsymbol{G}(\alpha)$, we have
	$$
	n\langle n\alpha\rangle\langle n\beta\rangle\leq\frac{1}{\log n}
	$$
	for infinitely many $n\in\N$.\footnote{In this paper, $\N$ denotes the set of positive integers.} Here, for a subset $X$ of $\R^k$, $\dim_H X$ denotes Hausdorff dimension of $X$. In \cite{EKL06}, M. Einsiedler, A. Katok and E. Lindenstrauss showed that the set of exceptions to Conjecture \ref{LittlewoodConjecture} is very small.
	
	\begin{thm}[{\cite[Theorem 1.5]{EKL06}}]\label{exception_Littlewood}
		The set of $(\alpha,\beta)\in\R^2$ such that {\rm (\ref{Littlewood_liminf})} does not hold has Hausdorff dimension zero. In fact, it is the countable union of compact subsets with box dimension zero.
	\end{thm}
	
	Lindenstrauss gives in \cite{Lin10} the explicit sufficient condition for $\alpha\in\R$ to satisfy (\ref{Littlewood_liminf}) for all $\beta\in\R$, which all $\alpha\in\R$ except on the set of Hausdorff dimension zero satisfy. It is obtained from the techniques in \cite{EKL06}.
	These results are the best for Conjecture \ref{LittlewoodConjecture} as of now.
	
	In this paper, we are interested in quantitative properties for Conjecture \ref{LittlewoodConjecture}. That is, for $(\alpha,\beta)\in\R^2$, we want to estimate the number of integers $n$ in $[1,N]$ such that $n\langle n\alpha\rangle\langle n\beta\rangle$ is small for large $N\in\N$. In regarding this problem, the quantitative version of the result of Pollington and Velani above is established in \cite{PVZZ22} (the equation (18)), which says that, for each $\alpha\in\Bad$ and $\gamma\in[0,1]$, there exists a subset $\boldsymbol{G}(\alpha,\gamma)\subset\R$ such that $\dim_H\boldsymbol{G}(\alpha,\gamma)=1$ and, for every $\beta\in\boldsymbol{G}(\alpha,\gamma)$, we have
	$$
	\left|\left\{n\in\{1,\dots,N\}\left|\ n\langle n\alpha\rangle\langle n\beta-\gamma\rangle\leq\frac{1}{\log n}\right.\right\}\right|\geq C\log\log N
	$$
	for $N\in\N$, where $C>0$ is a constant independent of $N$. Our main result is the quantitative version of Theorem \ref{exception_Littlewood}, which says that, for any $(\alpha,\beta)\in\R^2$ except on a set of small Hausdorff dimension, any small $0<\varepsilon<1$ and large $N\in\N$, the number of $1\leq n\leq N$ such that $n\langle n\alpha\rangle\langle n\beta\rangle<\varepsilon$ is larger than $\varepsilon\log N$ up to a constant.
	
	\begin{thm}\label{maintheorem}
		For any $0<\gamma<1/2$, there exist a subset $Z(\gamma)\subset\R^2$ and constants $C>0$ and $0<\varepsilon_0<1$ independent of $\gamma$ such that
		$\dim_HZ(\gamma)\leq 15\sqrt{\gamma}$ and, for any $(\alpha,\beta)\in \R^2\setminus Z(\gamma)$ and any $0<\varepsilon<\varepsilon_0$,
		$$
		\liminf_{N\to\infty}\frac{1}{\log N}\left|\left\{n\in\{1,\dots,N\}\left|\ n\langle n\alpha\rangle\langle n\beta\rangle<\varepsilon\right.\right\}\right|\geq C\gamma\varepsilon.
		$$
	\end{thm}
	
	By taking $Z=\bigcap_{0<\gamma<1/2}Z(\gamma)$, we have the following corollary.
	
	\begin{cor}\label{maincorollary}
		There exist $Z\subset\R^2$ and $0<\varepsilon_0<1$ such that $\dim_H Z=0$ and, for any $(\alpha,\beta)\in \R^2\setminus Z$ and any $0<\varepsilon<\varepsilon_0$,
		$$
		\liminf_{N\to\infty}\frac{1}{\log N}\left|\left\{n\in\{1,\dots,N\}\left|\ n\langle n\alpha\rangle\langle n\beta\rangle<\varepsilon\right.\right\}\right|\geq C(\alpha,\beta)\varepsilon
		$$
		for some constant $C(\alpha,\beta)>0$ determined by $(\alpha,\beta)$.
	\end{cor}
	
	We remark that we can not make $Z$ in Corollary \ref{maincorollary} to be contained in a countable union of compact subsets with box dimension zero as Theorem \ref{exception_Littlewood} by our argument. However, our main results give a lower bound of the number of $n\in[1,N]$ such that it makes $n\langle n\alpha\rangle\langle n\beta\rangle$ small which Theorem \ref{exception_Littlewood} does not give.
	In addition, our main results are applicable to much more $(\alpha,\beta)\in\R^2$ than the result in \cite{PVZZ22} and give a lower bound about $\log N$, greater than $\log\log N$.
	
	\subsection{The diagonal action on $\SL(n,\R)/\SL(n,\Z)$ and Littlewood's conjecture}
	
	Littlewood's conjecture is closely related to some action on $\SL(n,\R)/\SL(n,\Z)$ for $n\geq 3$, called {\it the diagonal action}. As we will explain below, Einsiedler, Katok and Lindenstrauss prove Theorem \ref{exception_Littlewood} from {\it the measure rigidity} of the diagonal action and our method to prove the main results is to prove some dynamical property of the diagonal action and apply it.
	
	We first give the definition of the diagonal action and the relation between it and Littlewood's conjecture.
	For $n\geq 3$, we write $G=\SL(n,\R)$, $\Gamma=\SL(n,\Z)$ and $X=G/\Gamma$.
	Since $\Gamma$ is a lattice of $G$ but not unimodular, $X$ has the unique Borel probability measure $m_X$ which is invariant under the action of $G$ and $X$ is not compact.
	We call $m_X$ the Haar measure on $X$.
	We write $A<G$ for the group of positive diagonal matrices with determinant one.
	The subgroup $A$ of $G$ acts on $X$ and we call this action the diagonal action on $X$. 
	Here we consider the case $n=3$. For $s,t\in\R$, we write
	$$
	a_{s,t}=
	\begin{pmatrix}
		e^{-s-t}&0&0\\0&e^s&0\\0&0&e^t
	\end{pmatrix}
	\in A.
	$$
	We define the subsemigroup $A^+$ of $A$ by $A^+=\left\{a_{s,t}\left|\ s,t\geq0\right.\right\}$.
	We write $U$ for the closed subgroup
	$$
	U=
	\left\{\left.u=
	\begin{pmatrix}
		1&0&0\\u_1&1&0\\u_2&0&1
	\end{pmatrix}
	\right|u_1,u_2\in\R\right\}
	$$
	of $G$ and, for $\alpha, \beta\in\R$,
	$$
	\tau_{\alpha,\beta}=
	\begin{pmatrix}
		1&0&0\\
		\alpha&1&0\\
		\beta&0&1
	\end{pmatrix}\in U.
	$$
	We notice that $U$ is unstable for the conjugation with $a$ for every $a\in A^+\setminus\{e\}$, where $e$ is the identity element of $G$. The following proposition shows the relation between the diagonal action and Littlewood's conjecture (see, for example, \cite[Proposition 11.1]{EKL06}).
	
	\begin{prop}\label{relations_diag_little}
		For $(\alpha,\beta)\in\R^2$, $(\alpha,\beta)$ satisfies {\rm (\ref{LittlewoodConjecture})} if and only if the orbit $A^+\tau_{\alpha,\beta}\Gamma$ is unbounded in $X$.
	\end{prop}
	
	Next, we give the extremely important fact on the diagonal action: the rigidity of invariant measures.
	It is conjectured that $A$-invariant and ergodic Borel probability measures on $X$ are very restricted. More precisely, the following is conjectured by G.A. Margulis.
	
	\begin{conj}\label{measurerigidity}
		Let $\mu$ be an $A$-invariant and ergodic Borel probability measure on $X=\SL(n,\R)/\SL(n,\Z)$ for $n\geq 3$. Then $\mu$ is algebraic, that is, there exists a closed subgroup $L<G$ such that $A<L$ and $\mu$ is the unique $L$-invariant Borel probability measure on a single, closed $L$-orbit on $X$.
	\end{conj}
	
	Conjecture \ref{measurerigidity} is still open.
	However, Einsiedler, Katok and Lindenstrauss proved the following important result toward Conjecture \ref{measurerigidity} in \cite{EKL06}. 
	We remark that, for $a\in A$ and an $a$-invariant Borel probability measure $\mu$ on $X$, we write $h_\mu(a)$ for the measure-theoretic entropy of the action of $a$ on $X$ with respect to $\mu$.
	
	\begin{thm}[{\cite[Theorem 1.3 and Corollary 1.4]{EKL06}}]\label{entropy_measurerigidity}
		Let $\mu$ be an $A$-invariant and ergodic Borel probability measure on $X=\SL(n,\R)/\SL(n,\Z)$ for $n\geq 3$. Assume that there exists $a\in A$ such that $h_\mu(a)>0$. Then $\mu$ is algebraic. In particular, if $n$ is prime, then $\mu$ is the Haar measure $m_X$ on $X$.
	\end{thm}
	
	Using Theorem \ref{entropy_measurerigidity} and Proposition \ref{relations_diag_little}, they proved Theorem \ref{exception_Littlewood}. We notice that Theorem \ref{entropy_measurerigidity} is also essential for the proof of our main results.
	
	\subsection{The empirical measures with respect to the diagonal action}\label{subsec_empirical_diagonal}
	
	We prove Theorem \ref{maintheorem} by studying {\it the empirical measures} with respect to the diagonal action. In the rest of this paper, we only consider the case of $n=3$, that is, $X=G/\Gamma=\SL(3,\R)/\SL(3,\Z)$.
	For $x\in X$ and $T>0$, we define the Borel probability measure $\delta^T_{A^+, x}$ on $X$ by
	$$
	\delta^T_{A^+,x}=\frac{1}{T^2}\int_{[0,T]^2}\delta_{a_{s,t}x}\ dsdt
	$$
	and call it the $T$-empirical measure of $x$ with respect to the action of $A^+$.
	Here, for $x\in X$, $\delta_x$ denotes the probability measure on $X$ supported by the one-point set $\{x\}$.
	For a sequence $\{\mu_k\}_{k=1}^\infty$ of Borel probability measures on $X$ and a finite Borel measure $\mu$ on $X$, we say that $\{\mu_k\}_{k=1}^\infty$ converges to $\mu$ if
	$$
	\int_Xf\ d\mu_k\xrightarrow[k\to\infty]{}\int_Xf\ d\mu
	$$
	for every $f\in C_0(X)$ and write $\mu_k\to\mu$. Then, it can be seen that $\mu(X)\leq 1$. Let $\widetilde{X}=X\sqcup\{\infty\}$ be the one-point compactification of $X$, then $\widetilde{X}$ is compact and metrizable. If we regard each $\mu_k$ as a Borel probability measure on $\widetilde{X}$ such that $\mu_k(\{\infty\})=0$, then the convergence of $\{\mu_k\}_{k=1}^\infty$ to $\mu$ is equivalent to that the sequence $\{\mu_k\}_{k=1}^\infty$ of Borel probability measures on $\widetilde{X}$ converges to the Borel probability measure $\mu+(1-\mu(X))\delta_\infty$ on $\widetilde{X}$ with respect to the weak* topology.
	For $x\in X$ and a finite Borel measure $\mu$ on $X$, we say that $\delta^T_{A^+,x}$ accumulates as $T\to\infty$ to $\mu$ if there exists a sequence $\{T_k\}_{k=1}^\infty\subset \R_{>0}$ such that $T_k\to\infty$ as $k\to\infty$ and $\{\delta^{T_k}_{A^+,x}\}_{k=1}^\infty$ converges to $\mu$. Then, it can be seen that $\mu$ is $A$-invariant.
	
	The following theorem plays the essential role for the proof of Theorem \ref{maintheorem}. Before we state the theorem, we prepare several notations. We write
	$$
	a_1=a_{1,0}=
	\begin{pmatrix}
		e^{-1}&0&0\\0&e&0\\0&0&1
	\end{pmatrix}
	\quad{\rm and}\quad
	a_2=a_{0,1}=
	\begin{pmatrix}
		e^{-1}&0&0\\0&1&0\\0&0&e
	\end{pmatrix}
	$$
	and these are the standard basis of $A^+$. We notice that $U$ is isomorphic to $\R^2$ through $\tau_{\alpha,\beta}$ and we define the metric on $U$ by $d_U(u,v)=\max\{|u_1-v_1|,|u_2-v_2|\}$. For $\rho>0$, we write $B^U_\rho=\left\{u\in U\left|\ d_U(e,u)<\rho\right.\right\}$.
	Then we state the theorem as follows.
	
	\begin{thm}\label{diagonalempirical}
		Let $x_0\in X$ and $0<\gamma<1$. We write $Z_{x_0}(\gamma)$ for the set of $u\in\overline{B^U_1}$ such that $\delta^T_{A^+,ux_0}$ accumulates as $T\to\infty$ to some $A$-invariant finite Borel measure $\mu$ on $X$ such that
		$$
		1-\gamma<\mu(X)\leq1\quad{\it and}\quad h_{\widehat\mu}(a_1)\leq \gamma,
		$$
		where $\widehat{\mu}=\mu(X)^{-1}\mu$ is the $A$-invariant Borel probability measure on $X$ obtained by normalizing $\mu$.
		Then we have
		$$
		\dim_HZ_{x_0}(\gamma)\leq15\sqrt{\gamma}.
		$$
	\end{thm}
	
	Theorem \ref{diagonalempirical} is the analogy of \cite[Theorem 1.5]{Usu22} for $\times a,\times b$ action on $\R/\Z$
	to the diagonal action on $X=\SL(3,\R)/\SL(3,\Z)$.
	For $0<\gamma<1/2$, $Z_{x_0}(\gamma)$ for $x_0=e\Gamma$ in Theorem \ref{diagonalempirical} corresponds to the set $Z(\gamma)$ of exceptions in Theorem \ref{maintheorem}. In the next Section \ref{empirical_measures_Littlewood}, we study closely the relation between the empirical measures with respect to the diagonal action and Littlewood's conjecture and see that Theorem \ref{maintheorem} actually holds for $(\alpha,\beta)$ such that $\tau_{\alpha,\beta}$ is not in $Z_{e\Gamma}(\gamma)$. In Section \ref{section_proof_crutialthm}, we prove Theorem \ref{diagonalempirical}.
	
	\subsection*{Acknowledgement}
	The author is grateful to Masayuki Asaoka for telling him Theorem \ref{converge_high_entrpy} below and pointing out the connection between his previous work \cite {Usu22} and Littlewood's conjecture. He also thanks to Mitsuhiro Shishikura for his helpful advice.
	This work is supported by JST SPRING, Grant Number JPMJSP2110.
	
	\section{Behavior of the empirical measures and Littlewood's conjecture}\label{empirical_measures_Littlewood}
	
	\subsection{The empirical measures of non-exceptional points}
	
	In this section, we study the relation between behavior of the empirical measures with respect to the diagonal action and Littlewood's conjecture. More precisely, we show that, for $0<\gamma<1/2$, $(\alpha,\beta)$ such that $\tau_{\alpha,\beta}$ is in the set $\overline{B^U_1}\setminus Z_{e\Gamma}(\gamma)$ in Theorem \ref{diagonalempirical} satisfies the conclusion of Theorem \ref{maintheorem} and prove Theorem \ref{maintheorem} assuming Theorem \ref{diagonalempirical}. Let $(\alpha,\beta)\in\R^2$ and $\{T_k\}_{k=1}^\infty$ in $\R_{>0}$ such that $T_k\to\infty$ as $k\to\infty$ and assume that the sequence $\{\delta^{T_k}_{A^+,\tau_{\alpha,\beta}\Gamma}\}_{k=0}^\infty$ of the empirical measures converges to some $A$-invariant finite Borel measure $\mu$ on $X$. If $\tau_{\alpha,\beta}\in\overline{B^U_1}\setminus Z_{e\Gamma}(\gamma)$, we have the two cases:
	\begin{enumerate}
		\renewcommand{\labelenumi}{\rm{(\roman{enumi})}}
		\item $\mu$ satisfies $1-\gamma<\mu(X)\leq 1$ and $h_{\widehat{\mu}}(a_1)>\gamma$ for
		$\widehat{\mu}=\mu(X)^{-1}\mu$,
		\item $\mu$ satisfies $\mu(X)\leq 1-\gamma$.
	\end{enumerate}
	We call the case (i) {\it large entropy case}.
	The case (ii) is the case of, what we call, {\it escape of mass}.
	
	\begin{dfn}\label{def_escape_of_mass}
		Let $\{\mu_k\}_{k=1}^\infty$ be a sequence of Borel probability measures on $X$ and $0<\gamma\leq 1$. We say that $\{\mu_k\}_{k=1}^\infty$ exhibits $\gamma$-escape of mass if
		$$
		\limsup_{k\to\infty}\mu_k(K)\leq 1-\gamma
		$$
		for any compact subset $K$ of $X$.
	\end{dfn}
	
	Corresponding to (i) and (ii), the following Theorem \ref{converge_high_entrpy} and Theorem \ref{escape_of_mass} of quantitative results for Conjecture \ref{LittlewoodConjecture} are established.
	
	\begin{thm}\label{converge_high_entrpy}
		There exist constants $C_1>0$ and $0<\varepsilon_0<1/2$ which satisfy the following. Let $0<\gamma<1$, $\alpha,\beta\in\R$ and $\{T_k\}_{k=1}^\infty$ be a sequence in $\R_{>0}$ such that $T_k\to\infty$ as $k\to\infty$. Assume that, for $\tau_{\alpha,\beta}\Gamma\in X$, $\{\delta^{T_k}_{A^+,\tau_{\alpha,\beta}\Gamma}\}_{k=1}^\infty$ converges to an $A$-invariant finite Borel measure $\mu$ such that $1-\gamma<\mu(X)\leq 1$ and $h_{\widehat{\mu}}(a_1)>\gamma$.
		Then, for every $0<\varepsilon<\varepsilon_0$, we have
		$$
		\liminf_{k\to\infty}\frac{1}{T_k}\left|\left\{n\in \N\left|\ n<e^{2T_k}, n\langle n\alpha\rangle\langle n\beta\rangle<\varepsilon\right.\right\}\right|\geq C_1(1-\gamma)\gamma\varepsilon.
		$$
	\end{thm}
	
	\begin{thm}\label{escape_of_mass}
		Let $0<\gamma<1$, $\alpha,\beta\in\R$ and $\{T_k\}_{k=1}^\infty$ be a sequence in $\R_{>0}$ such that $T_k\to\infty$ as $k\to\infty$. Assume that, for $\tau_{\alpha,\beta}\Gamma\in X$, $\{\delta^{T_k}_{A^+,\tau_{\alpha,\beta}\Gamma}\}_{k=1}^\infty$ exhibits $\gamma$-escape of mass. Then, for every $0<\varepsilon<1/2$, we have
		$$
		\liminf_{k\to\infty}\frac{1}{T_k}\left|\left\{n\in \N\left|\ n<e^{2T_k}, n\langle n\alpha\rangle\langle n\beta\rangle<\varepsilon\right.\right\}\right|\geq\frac{\gamma}{3\log2}.
		$$
	\end{thm}
	
	We prove Theorem \ref{converge_high_entrpy} and Theorem \ref{escape_of_mass} in Subsection \ref{subsec_high_entropy} and \ref{subseq_escape_of_mass}, respectively.
	Here we show that Theorem \ref{maintheorem} is established if we assume that Theorem \ref{converge_high_entrpy}, Theorem \ref{escape_of_mass} and Theorem \ref{diagonalempirical} hold.
	
	\begin{proof}[Proof of Theorem \ref{maintheorem}]
		Assume Theorem \ref{converge_high_entrpy}, Theorem \ref{escape_of_mass} and Theorem \ref{diagonalempirical} hold.
		Let $C_1>0,0<\varepsilon_0<1/2$ be constants in Theorem \ref{converge_high_entrpy} and $0<\gamma<1/2$.
		We define $Z(\gamma)\subset [-1,1]^2$ as the set of $(\alpha,\beta)\in[-1,1]^2$ such that $\tau_{\alpha,\beta}\in Z_{e\Gamma}(\gamma)$, where $Z_{e\Gamma}(\gamma)$ is the subset of $\overline{B^U_1}$ in Theorem \ref{diagonalempirical} for $e\Gamma\in X$. Since $\R^2$ and $U$ are isometric through $\tau_{\alpha,\beta}$, we have $\dim_H Z(\gamma)=\dim_H Z_{e\Gamma}(\gamma)\leq 15\sqrt{\gamma}$. Let $(\alpha,\beta)\in[-1,1]^2\setminus Z(\gamma)$ and $0<\varepsilon<\varepsilon_0$. We show that
		\begin{equation}\label{quant_little}
			\liminf_{T\to\infty}\frac{1}{T}\left|\left\{n\in\N\left|\ n<e^{2T}, n\langle n\alpha\rangle\langle n\beta\rangle<\varepsilon\right.\right\}\right|\geq \min\left\{C_1(1-\gamma)\gamma\varepsilon,\frac{\gamma}{3\log2}\right\}.
		\end{equation}
		
		Suppose the inequality (\ref{quant_little}) does not hold. Then there exist a sequence $\{T_k\}_{k=1}^\infty$ in $\R_{>0}$ and $0<\eta<1$ such that $T_k\to\infty$ as $k\to\infty$ and
		\begin{equation}\label{contradiction_quant_little}
			\frac{1}{T_k}\left|\left\{n\in\N\left|\ n<e^{2T_k}, n\langle n\alpha\rangle\langle n\beta\rangle<\varepsilon\right.\right\}\right|\leq \min\left\{C_1(1-\gamma)\gamma\varepsilon,\frac{\gamma}{3\log2}\right\}-\eta
		\end{equation}
		for all $k$. Since the space of Borel probability measures on $\widetilde{X}=X\sqcup\{\infty\}$ is compact with respect to the weak* topology, by taking some subsequence, we can assume that the sequence $\{\delta^{T_k}_{A^+,\tau_{\alpha,\beta}\Gamma}\}_{k=1}^\infty$ of the empirical measures converges to some $A$-invariant finite Borel measure $\mu$ on $X$. Since $\tau_{\alpha,\beta}\in \overline{B^U_1}\setminus Z_{e\Gamma}(\gamma)$, we have $\mu(X)\leq 1-\gamma$ or $h_{\widehat{\mu}}(a_1)>\gamma$. If $\mu(X)\leq 1-\gamma$, then it follows that $\{\delta^{T_k}_{A^+,\tau_{\alpha,\beta}\Gamma}\}_{k=1}^\infty$ exhibits $\gamma$-escape of mass and, from Theorem \ref{escape_of_mass},
		$$
		\liminf_{k\to\infty}\frac{1}{T_k}\left|\left\{n\in \N\left|\ n<e^{2T_k}, n\langle n\alpha\rangle\langle n\beta\rangle<\varepsilon\right.\right\}\right|\geq\frac{\gamma}{3\log2}.
		$$
		This contradicts (\ref{contradiction_quant_little}). Hence, we have $1-\gamma<\mu(X)\leq 1$ and
		$h_{\widehat{\mu}}(a_1)>\gamma$. However, from Theorem \ref{converge_high_entrpy}, it follows that
		$$
		\liminf_{k\to\infty}\frac{1}{T_k}\left|\left\{n\in \N\left|\ n<e^{2T_k}, n\langle n\alpha\rangle\langle n\beta\rangle<\varepsilon\right.\right\}\right|\geq C_1(1-\gamma)\gamma\varepsilon
		$$
		and we have again a contradiction to (\ref{contradiction_quant_little}). Therefore, we showed the inequality (\ref{quant_little}).
		
		We can take a constant $C$ such that $\min\left\{C_1(1-\gamma)\gamma\varepsilon, \gamma/3\log2\right\}/2\geq C\gamma\varepsilon$ for any $0<\varepsilon<1/2$ and $0<\gamma<1/2$. By putting $T=\log N/2$ in (\ref{quant_little}), we have
		$$
		\liminf_{N\to\infty}\frac{1}{\log N}\left|\left\{n\in\{1,\dots,N-1\}\left|\ n\langle n\alpha\rangle\langle n\beta\rangle<\varepsilon\right.\right\}\right|\geq C\gamma\varepsilon
		$$
		and obtain the conclusion of Theorem \ref{maintheorem}. By replacing $Z(\gamma)$ to $\bigcup_{(k,l)\in \Z^2}(Z(\gamma)+(k,l))$, we obtain Theorem \ref{maintheorem}.
	\end{proof}
	
	\subsection{The empirical measures of large entropy case}\label{subsec_high_entropy}
	
	We prove Theorem \ref{converge_high_entrpy}. We need the following proposition about the representation of entropy for generalized convex combinations of invariant measures.
	
	\begin{prop}\label{entropyrep}
		Let $Y$ be a compact metric space and $T:Y\to Y$ be a continuous map. We write $M_T(Y)$ for the space of all $T$-invariant Borel probability measures on $Y$ with the weak* topology.
		Suppose $\mu\in M_T(Y)$ and $\tau$ to be a Borel probability measure on $M_T(Y)$ such that
		\begin{equation*}
			\mu=\int_{M_T(Y)} m\ d\tau(m),
		\end{equation*}
		that is,
		$$
		\int_Yf\ d\mu=\int_{M_T(Y)}\int_Yf\ dm\ d\tau(m)
		$$
		for all $f\in C(Y)$. Then we have
		$$
		h_\mu(T)=\int_{M_T(Y)}h_m(T)\ d\tau(m).
		$$
	\end{prop}
	
	This proposition is a generalization of \cite[Theorem 8.4]{Wal82} to any generalized convex combination of invariant measures and we can prove it in the same way as \cite[Theorem 8.4]{Wal82} by applying \cite[Lemma 10.7]{Phe01} to any generalized convex combination. We notice that we can extend the diagonal action on $X$ to the continuous action of $A$ on $\widetilde{X}=X\sqcup\{\infty\}$ with a fixed point $\infty$ and consider an $A$-invariant finite Borel measure $\mu$ on $X$ with $\mu(X)\leq1$ as an $A$-invariant Borel probability measure on $\widetilde{X}$ by letting $\mu(\{\infty\})=1-\mu(X)$. 
	Since $\widetilde{X}$ is compact and metrizable, we can apply Proposition \ref{entropyrep} to the action of $a$ on $\widetilde{X}$ for $a\in A$.
	If $\mu$ is $a$-invariant, then we have $h_\mu(a)=(1-\mu(X))h_{\delta_\infty}(a)+\mu(X)h_{\widehat{\mu}}(a)=\mu(X)h_{\widehat{\mu}}(a)$, where the left hand side is the measure-theoretic entropy of the action of $a$ on $\widetilde{X}$ with respect to $\mu$ and, as above, $\widehat{\mu}$ in the right hand side is the Borel probability measure on $X$ obtained by normalizing $\mu$.
	
	\begin{proof}[Proof of Theorem \ref{converge_high_entrpy}]
		Let $0<\gamma<1$, $\alpha,\beta\in\R$, $\{T_k\}_{k=1}^\infty\subset\R_{>0}$ and $\mu$ satisfy the assumption in Theorem \ref{converge_high_entrpy}. First, we take the ergodic decomposition of $\mu$ with respect to the diagonal action, that is, we regard $\mu$ as an $A$-invariant Borel probability measure on $\widetilde{X}$ and take the ergodic decomposition of $\mu$ with respect to the continuous action of $A$ on the compact and metrizable space $\widetilde{X}$. Let $E_A(X)$ (resp. $E_A(\widetilde{X})$) be the set of $A$-invariant and ergodic Borel probability measures on $X$ (resp. $\widetilde{X}$). Then there exists a Borel probability measure $\widetilde{\tau}$ on $E_A(\widetilde{X})$ such that
		$$
		\mu=\int_{E_A(\widetilde{X})}m\ d\widetilde{\tau}(m).
		$$
		on $\widetilde{X}$. Since $\infty\in\widetilde{X}$ is a fixed point of the action of $A$, we have
		\begin{equation}\label{ergodicdecomposition}
			\mu=(1-\mu(X))\delta_\infty+\mu(X)\int_{E_A(X)}m\ d\tau(m),
		\end{equation}
		where $\tau$ is the Borel probability measure on $E_A(X)$ obtained by normalizing the restriction of $\widetilde{\tau}$ to $E_A(X)$. We apply Proposition \ref{entropyrep} to (\ref{ergodicdecomposition}) and the action of $a_1$ on $\widetilde{X}$ and obtain
		$$
		h_\mu(a_1)=\mu(X)h_{\widehat{\mu}}(a_1)=\mu(X)\int_{E_A(X)}h_m(a_1)\ d\tau(m)
		$$
		and, by dividing both sides by $\mu(X)>1-\gamma>0$, 
		$$
		h_{\widehat{\mu}}(a_1)=\int_{E_A(X)}h_m(a_1)\ d\tau(m).
		$$
		By our assumption, we have $h_{\widehat{\mu}}(a_1)>\gamma$.
		Furthermore, from Theorem \ref{entropy_measurerigidity}, it follows that $h_m(a_1)=0$ for all $m\in E_A(X)\setminus\{m_X\}$ and we have $h_{m_X}(a_1)=4$ (see \cite[Proposition 9.6]{MT94} or \cite[Lemma 6.2]{EK03}). Hence, the right hand side is $4\tau(\{m_X\})$ and we obtain
		$$
		\tau(\{m_X\})>\frac{\gamma}{4}.
		$$
		This inequality, our assumption that $\mu(X)>1-\gamma$, and (\ref{ergodicdecomposition}) imply that
		\begin{equation}\label{proportion_haar}
			\mu(E)\geq \frac{(1-\gamma)\gamma}{4}m_X(E)
		\end{equation}
		for all Borel subset $E$ of $X$.
		
		By our assumption, $\delta^{T_k}_{A^+,\tau_{\alpha,\beta}\Gamma}\to\mu$ as $k\to\infty$. Then it follows from the inequality (\ref{proportion_haar}) that
		\begin{align}\label{semiequidistribution}
			\liminf_{k\to\infty}\frac{1}{T_k^2}m_{\R^2}\left(\left\{(s,t)\in[0,T_k]^2\left|\ a_{s,t}\tau_{\alpha,\beta}\Gamma\in O\right.\right\}\right)=&\ \liminf_{k\to\infty}\delta^{T_k}_{A^+,\tau_{\alpha,\beta}\Gamma}(O)
			\nonumber\\
			\geq&\ \mu (O)\nonumber\\
			\geq&\ \frac{(1-\gamma)\gamma}{4}m_X(O)
		\end{align}
		for any open subset $O$ of $X$, where $m_{\R^2}$ is the Lebesgue measure on $\R^2$.
		We notice that we can consider $X$ as the space of unimodular lattices in $\R^3$, that is, discrete additive subgroups of $\R^3$ with the covolume $1$, by identifying $g\Gamma\in X\ (g\in G)$ with $g\cdot \Z^3\subset \R^3$. We take small constants $0<\eta_0,\varepsilon_0<1$ which are determined only by $X$.  For $0<\varepsilon<\varepsilon_0$, we define the open subset $V_\varepsilon$ of $X$ by
		\begin{equation*}
			V_\varepsilon=\left\{x\in X\left|\ x\subset\R^3\ {\rm intersects\ with}\ (e^{-\eta_0},1)\times(-\sqrt{\varepsilon},\sqrt{\varepsilon})^2\right.\right\}.
		\end{equation*}
		
		Let $k\in\{1,2,\dots\}$ and $(s,t)\in[0,T_k]^2$ and suppose $a_{s,t}\tau_{\alpha,\beta}\Gamma\in V_\varepsilon$. Then there exists $(n,m_1,m_2)\in\Z^3\setminus\{0\}$ such that
		\begin{equation*}
			a_{s,t}\tau_{\alpha,\beta}
			\begin{pmatrix}
				n\\m_1\\m_2
			\end{pmatrix}
			=\begin{pmatrix}
				e^{-s-t}n\\e^s(n\alpha+m_1)\\e^t(n\beta+m_2)
			\end{pmatrix}
			\in(e^{-\eta_0},1)\times (-\sqrt{\varepsilon},\sqrt{\varepsilon})^2.
		\end{equation*}
		Hence, we have that $n\in\N$, 
		\begin{equation*}\label{littlewood_ineq}
			n\langle n\alpha\rangle\langle n\beta\rangle\leq n|n\alpha+m_1||n\beta+m_2|<\varepsilon
		\end{equation*}
		and
		\begin{equation}\label{estimate_n}
			\log n<s+t<\log n+\eta_0.
		\end{equation}
		By (\ref{estimate_n}), we have $n<e^{s+t}\leq e^{2T_k}$. Here we write
		$$
		\Lambda_\varepsilon(T_k)=\left\{n\in\N\left|\ n<e^{2T_k}, n\langle n\alpha\rangle\langle n\beta\rangle<\varepsilon\right.\right\}.
		$$
		Then, for $(s,t)\in[0,T_k]^2$ such that $a_{s,t}\tau_{\alpha,\beta}\Gamma\in V_\varepsilon$ and $n\in\N$ which is taken above for $(s,t)$, we have $n\in\Lambda_\varepsilon(T_k)$ and hence, by (\ref{estimate_n}), 
		\begin{equation*}
			s+t\in\bigcup_{n\in\Lambda_\varepsilon(T_k)}\left(\log n,\log n+\eta_0\right).
		\end{equation*}
		Then it follows that
		\begin{equation}\label{thenumber_n}
			m_{\R}\left(\left\{s+t\left|\ (s,t)\in[0,T_k]^2, a_{s,t}\tau_{\alpha,\beta}\Gamma\in V_\varepsilon\right.\right\}\right)\leq \eta_0\cdot|\Lambda_\varepsilon(T_k)|,
		\end{equation}
		where $m_\R$ is the Lebesgue measure on $\R$.
		About the left hand side, by changing variables and Fubini's theorem, we have
		\begin{align*}
			&m_{\R^2}\left(\left\{(s,t)\in[0,T_k]^2\left|\ a_{s,t}\tau_{\alpha,\beta}\Gamma\in V_\varepsilon\right.\right\}\right)\\
			=&\int_0^{T_k}\int_{t'}^{t'+T_k}\chi_{\left\{(s,t)\in[0,T_k]^2\left|\ a_{s,t}\tau_{\alpha,\beta}\Gamma\in V_\varepsilon\right.\right\}}(s'-t',t')\ ds'dt'\\
			\leq&\int_0^{T_k}\int_{\R}\chi_{\left\{s+t\ \left|\ (s,t)\in[0,T_k]^2, a_{s,t}\tau_{\alpha,\beta}\Gamma\in V_\varepsilon\right.\right\}}(s')\ ds'dt'\\
			=&\ T_k\cdot m_{\R}\left(\left\{s+t\left|\ (s,t)\in[0,T_k]^2, a_{s,t}\tau_{\alpha,\beta}\Gamma\in V_\varepsilon\right.\right\}\right),
		\end{align*}
		where $\chi_\cdot$ is the characteristic function for a set.
		From this and the inequality (\ref{thenumber_n}), we have
		\begin{equation*}
			|\Lambda_\varepsilon(T_k)|\geq\frac{1}{\eta_0T_k}m_{\R^2}\left(\left\{(s,t)\in[0,T_k]^2\left|\ a_{s,t}\tau_{\alpha,\beta}\Gamma\in V_\varepsilon\right.\right\}\right).
		\end{equation*}
		By dividing by $T_k$ and letting $k\to\infty$, we have from the inequality (\ref{semiequidistribution}) that
		\begin{align}\label{liminf_bound}
			\liminf_{k\to\infty}\frac{|\Lambda_\varepsilon(T_k)|}{T_k}\geq&\ \frac{1}{\eta_0}\liminf_{k\to\infty}\frac{1}{T_k^2}m_{\R^2}\left(\left\{(s,t)\in[0,T_k]^2\left|\ a_{s,t}\tau_{\alpha,\beta}\Gamma\in V_\varepsilon\right.\right\}\right)\nonumber\\
			=&\ \frac{(1-\gamma)\gamma}{4\eta_0}m_X(V_\varepsilon).
		\end{align}
		Finally, we estimate $m_X(V_\varepsilon)$.
		We take $0<\eta_0,\varepsilon_0<1$ in the definition $V_\varepsilon$ sufficiently small so that $W=\left\{g=(g_{i,j})_{1\leq i,j\leq 3}\in G\left|\ e^{-\eta_0}<g_{i,i}<e^{\eta_0},\ |g_{i,j}|< \sqrt{\varepsilon_0}\ (i\neq j)\right.\right\}$ is a small open neighborhood and $\pi_X: W\ni g\mapsto g\Gamma\in X$ is injective. 
		Then $m_X|_{\pi_X(W)}$ equals ${\pi_X}_*m_G|_W$ up to a constant, where $m_G$ is the Haar measure on $G$.
		If we write $V'_\varepsilon=\left\{g\in W\left|\ e^{-\eta_0}<g_{1,1}<1,\ |g_{1,2}|,|g_{1,3}|< \sqrt{\varepsilon}\right.\right\}$ for $0<\varepsilon<\varepsilon_0$, then we can see that
		$V'_\varepsilon\subset W$ and $\pi_X(V'_\varepsilon)\subset V_\varepsilon$.
		Since the Haar measure $m_G$ on $G$ is smooth, we have
		$m_X(\pi_X(V'_\varepsilon))\geq C(\sqrt{\varepsilon})^2=C\varepsilon$ for some constant $C>0$ independent of $\varepsilon$. Hence, we have $m_X(V_\varepsilon)\geq m_X(\pi_X(V'_\varepsilon))\geq C\varepsilon$. This estimate and the inequality (\ref{liminf_bound}) imply
		$$
		\liminf_{k\to\infty}\frac{|\Lambda_\varepsilon(T_k)|}{T_k}\geq \frac{C}{4\eta_0}(1-\gamma)\gamma\varepsilon,
		$$
		and we complete the proof.
	\end{proof}
	
	\subsection{The empirical measures of the case of escape of mass}\label{subseq_escape_of_mass}
	
	We prove Theorem \ref{escape_of_mass}. We can obtain Theorem \ref{escape_of_mass} from the following Proposition \ref{cusp_littlewood}. We notice that, as we saw in the proof of Theorem \ref{converge_high_entrpy}, $X$ can be regarded as the space of unimodular lattices in $\R^3$. We take $0<\varepsilon<1/2$ and write $B^{\R^3}_\varepsilon=\left\{\boldsymbol{x}\in\R^3\left|\|\boldsymbol{x}\|_\infty<\varepsilon\right.\right\}$,
	where, for $\boldsymbol{x}={}^t(x_1,x_2,x_3)\in\R^3$, $\|\boldsymbol{x}\|_\infty=\max\{|x_1|,|x_2|,|x_3|\}$. For $\varepsilon$, we define the subset $X_{\varepsilon}$ by
	$$
	X_{\varepsilon}=\left\{x\in X\left|\ x\cup\overline{B^{\R^3}_\varepsilon}\neq\{0\}\right.\right\}.
	$$
	Then we state Proposition \ref{cusp_littlewood} as follows.
	
	\begin{prop}\label{cusp_littlewood}
		Let $\alpha,\beta\in\R\setminus\Q$. For $0<\gamma<1$, $0<\varepsilon<1/2$ and $T>0$, we assume that
		\begin{equation}\label{escape_to_cusp}
			\frac{1}{T^2}m_{\R^2}\left(\left\{(s,t)\in[0,T]^2\left|\ a_{s,t}\tau_{\alpha,\beta}\Gamma\in X_{\varepsilon}\right.\right\}\right)\geq \gamma.
		\end{equation}
		Then we have
		$$
		\frac{1}{T}\left|\left\{n\in\N\left|\ n<e^{2T}, n\langle n\alpha\rangle\langle n\beta\rangle\leq\varepsilon^3\right.\right\}\right|\geq \frac{\gamma}{3\log 2}.
		$$
	\end{prop}
	
	Before we prove Proposition \ref{cusp_littlewood}, we see that Theorem \ref{escape_of_mass} follows from Proposition \ref{cusp_littlewood}.
	
	\begin{proof}[Proof of Theorem \ref{escape_of_mass}]
		Assume Proposition \ref{cusp_littlewood} holds. Let $0<\gamma<1$, $\alpha,\beta\in\R$ and $\{T_k\}_{k=1}^\infty$ be a sequence in $\R_{>0}$ such that $T_k\to\infty$ as $k\to\infty$ and assume that $\{\delta^{T_k}_{A^+,\tau_{\alpha,\beta}\Gamma}\}_{k=1}^\infty$ exhibits $\gamma$-escape of mass.
		If $\alpha\in\Q$ or $\beta\in\Q$, then the conclusion is trivial. Hence, we assume that $\alpha,\beta\in\R\setminus\Q$.
		We take $0<\varepsilon<1/2$ and define the subset $X'_{\varepsilon}$ of $X$ by
		$$
		X'_{\varepsilon}=\left\{x\in X\left|\ x\cup B^{\R^3}_\varepsilon\neq\{0\}\right.\right\}.
		$$
		Then $X'_{\varepsilon}\subset X_\varepsilon$, $X'_{\varepsilon}$ is open and, by Mahler's compactness criterion (see \cite[Theorem 11.33]{EW11}), $X\setminus X'_{\varepsilon}$ is compact. Since $\{\delta^{T_k}_{A^+,\tau_{\alpha,\beta}\Gamma}\}_{k=1}^\infty$ exhibits $\gamma$-escape of mass, it follows that
		\begin{align*}
			&\liminf_{k\to\infty}\frac{1}{T_k^2}m_{\R^2}\left(\left\{(s,t)\in[0,T_k]^2\left|\ a_{s,t}\tau_{\alpha,\beta}\Gamma\in X_\varepsilon\right.\right\}\right)\nonumber\\
			\geq&\ \liminf_{k\to\infty}\frac{1}{T_k^2}m_{\R^2}\left(\left\{(s,t)\in[0,T_k]^2\left|\ a_{s,t}\tau_{\alpha,\beta}\Gamma\in X'_\varepsilon\right.\right\}\right)\nonumber\\
			=&\ \liminf_{k\to\infty}\delta^{T_k}_{A^+,\tau_{\alpha,\beta}\Gamma}(X'_\varepsilon)
			\nonumber\\
			\geq&\ \gamma.
		\end{align*}
		Hence, if we take arbitrary $0<\eta<\gamma$, then we have for sufficiently large $k$ that
		$$
		\frac{1}{T_k^2}m_{\R^2}\left(\left\{(s,t)\in[0,T_k]^2\left|\ a_{s,t}\tau_{\alpha,\beta}\Gamma\in X_\varepsilon\right.\right\}\right)\geq\gamma-\eta
		$$
		and we can apply Proposition \ref{cusp_littlewood} to this inequality. Therefore, we obtain
		\begin{align*}
			\frac{1}{T_k}\left|\left\{n\in\N\left|\ n<e^{2T_k}, n\langle n\alpha\rangle\langle n\beta\rangle<\varepsilon\right.\right\}\right|\geq&
			\ \frac{1}{T_k}\left|\left\{n\in\N\left|\ n<e^{2T_k}, n\langle n\alpha\rangle\langle n\beta\rangle\leq\varepsilon^3\right.\right\}\right|\\
			\geq& \ \frac{\gamma-\eta}{3\log 2}.
		\end{align*}
		By letting $k\to\infty$ and $\eta\to 0$, we obtain
		$$
		\liminf_{k\to\infty}\frac{1}{T_k}\left|\left\{n\in \N\left|\ n<e^{2T_k}, n\langle n\alpha\rangle\langle n\beta\rangle<\varepsilon\right.\right\}\right|\geq\frac{\gamma}{3\log2}
		$$
		and complete the proof.
	\end{proof}
	
	We prove Proposition \ref{cusp_littlewood} in the rest of this subsection.
	Let $\alpha,\beta\in\R\setminus\Q$, $0<\gamma<1$, $0<\varepsilon<1/2$ and $T>0$ satisfy (\ref{escape_to_cusp}).
	For $\boldsymbol{n}\in\Z^3\setminus\{0\}$, we define
	$$
	d_{\varepsilon,\boldsymbol{n}}=\left\{(s,t)\in\R^2\left|\ \|a_{s,t}\tau_{\alpha,\beta}\boldsymbol{n}\|_\infty\leq\varepsilon\right.\right\}.
	$$
	If $d_{\varepsilon,\boldsymbol{n}}\neq \emptyset$
	for  $\boldsymbol{n}={}^t(n,m_1,m_2)\in\Z^3\setminus\{0\}$, then $n\neq0$,
	$$
	d_{\varepsilon,\boldsymbol{n}}=\left\{(s,t)\in\R^2\left|\ s\leq\log\frac{\varepsilon}{|n\alpha+m_1|},\ t\leq\log\frac{\varepsilon}{|n\beta+m_2|},\ s+t\geq\log\frac{|n|}{\varepsilon}\right.\right\}
	$$
	and this is the isosceles right triangle with the length of the leg $\log (\varepsilon^3/|n||n\alpha+m_1||n\beta+m_2|)$. We notice that, since $\alpha,\beta\in\R\setminus\Q$, we have $|n\alpha+m_1|,|n\beta+m_2|>0$.
	\begin{figure}[b]
		\centering
		\includegraphics[width=10cm]{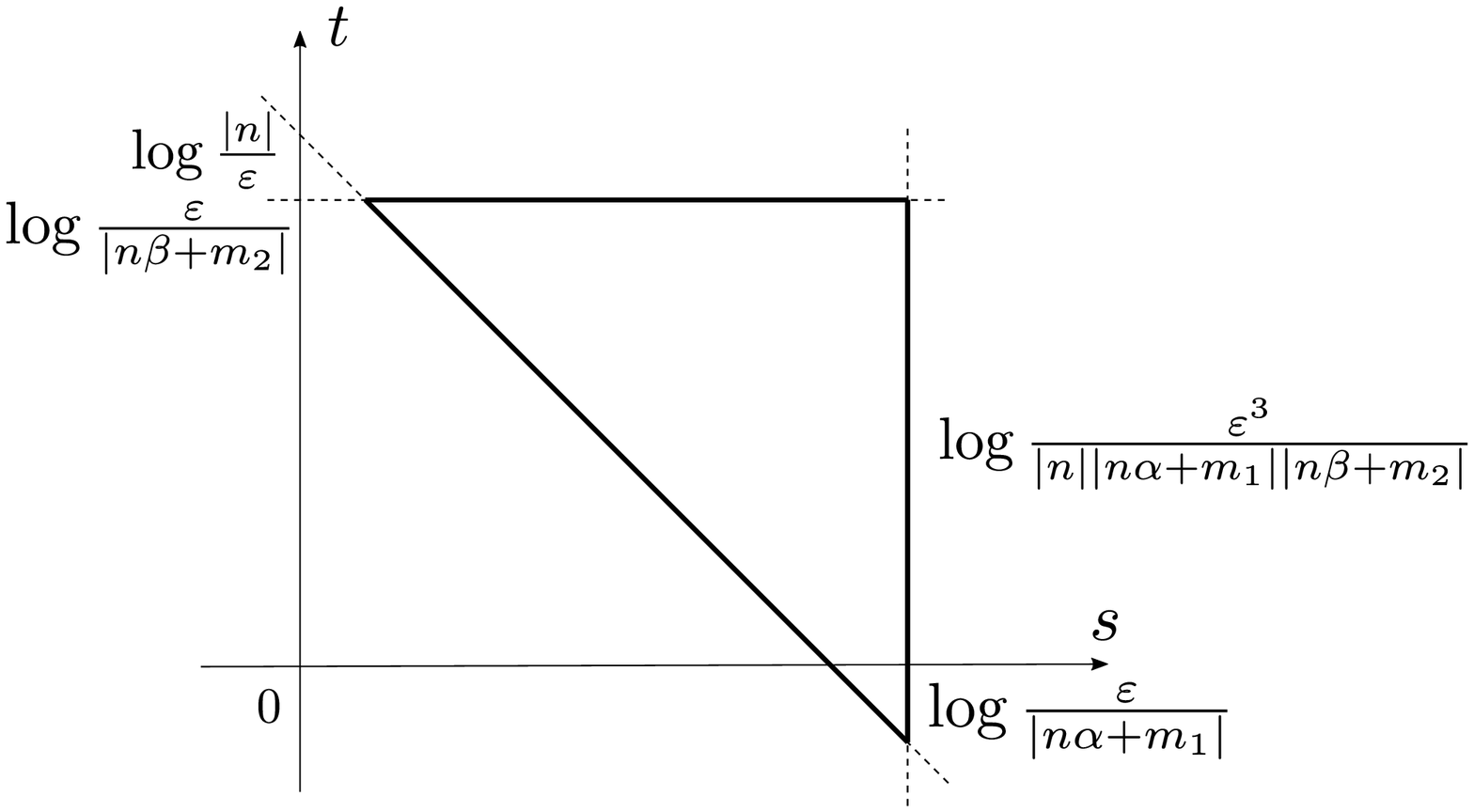}
		\caption{$d_{\varepsilon,\boldsymbol{n}}$}
	\end{figure}
	By the definition of $X_\varepsilon$, it follows that
	\begin{equation}\label{cover_by_d}
		\left\{(s,t)\in[0,T]^2\left|\ a_{s,t}\tau_{\alpha,\beta}\Gamma\in X_{\varepsilon}\right.\right\}=\bigcup_{\boldsymbol{n}\in\Z^3\setminus\{0\}}\left(d_{\varepsilon,\boldsymbol{n}}\cap[0,T]^2\right).
	\end{equation}
	Since $d_{\varepsilon,\boldsymbol{n}}=d_{\varepsilon,-\boldsymbol{n}}$ for each $\boldsymbol{n}\in\Z^3\setminus\{0\}$, we take the union of the right hand side only over $\boldsymbol{n}={}^t(n,m_1,m_1)\in\Z^3\setminus\{0\}$ such that $n>0$. Here we see the following lemma.
	
	\begin{lem}
		For each $n\in\N$, the element $\boldsymbol{n}={}^t(n,m_1,m_2)$ of $\Z^3\setminus\{0\}$ such that
		$d_{\varepsilon, \boldsymbol{n}}\cap[0,T]^2\neq \emptyset$ is at most only one.
	\end{lem}
	
	\begin{proof}
		Let $n\in\N$ such that $d_{\varepsilon,\boldsymbol{n}}\cap[0,T]^2\neq\emptyset$ for some $\boldsymbol{n}={}^t(n,m_1,m_2)$. Then, for $(s,t)\in d_{\varepsilon,\boldsymbol{n}}\cap[0,T]^2$, we have $|n\alpha+m_1|\leq e^{-s}\varepsilon\leq\varepsilon$ and $|n\beta+m_2|\leq e^{-t}\varepsilon\leq\varepsilon$. Hence, for any $\boldsymbol{n}'={}^t(n,m'_1,m'_2)\in\Z^3\setminus\{0\}$ such that $\boldsymbol{n}'\neq\boldsymbol{n}$, we have $m'_i\neq m_i$ for $i=1$ or $2$ and, if $i=1$, $|n\alpha+m'_1|\geq |m'_1-m_1|-|n\alpha+m_1|\geq 1-\varepsilon$. Then we have $\log(\varepsilon/|n\alpha+m'_1|)\leq\log(\varepsilon/(1-\varepsilon))<0$ and $d_{\varepsilon,\boldsymbol{n}'}\cap[0,T]^2=\emptyset$. We have the same conclusion when $i=2$.
	\end{proof}
	
	For each $n\in\N$, we write $d_{\varepsilon,n}=d_{\varepsilon,\boldsymbol{n}}$ if there exists $\boldsymbol{n}={}^t(n,m_1,m_2)\in\Z^3$ such that $d_{\varepsilon, \boldsymbol{n}}\cap[0,T]^2\neq \emptyset$,
	and $d_{\varepsilon,n}=\emptyset$ if not. Then we can write (\ref{cover_by_d}) as
	\begin{equation}\label{cover_by_dn}
		\left\{(s,t)\in[0,T]^2\left|\ a_{s,t}\tau_{\alpha,\beta}\Gamma\in X_{\varepsilon}\right.\right\}=\bigcup_{n\in\N}\left(d_{\varepsilon,n}\cap[0,T]^2\right)
	\end{equation}
	and we notice that, for each $n\in\N$, the isosceles right triangle $d_{\varepsilon,\boldsymbol{n}}$ with its hypotenuse on the line $\{(s,t)\in\R^2\left|s+t=\log(n/\varepsilon)\right.\}$ is at most only one.
	
	We define $\pi:\R^2\to\R$ by $\pi(s,t)=s+t$. 
	We notice that, for each $n\in\N$ such that $d_{\varepsilon,n}\cap[0,T]^2\neq\emptyset$, $\pi\left(d_{\varepsilon,n}\cap[0,T]^2\right)$ is the interval $[\log(n/\varepsilon),s_n+t_n]$, where $(s_n,t_n)\in d_{\varepsilon,n}\cap[0,T]^2$ (see Figure \ref{cover_dn}).
	\begin{figure}[b]
		\centering
		\includegraphics[width=12cm]{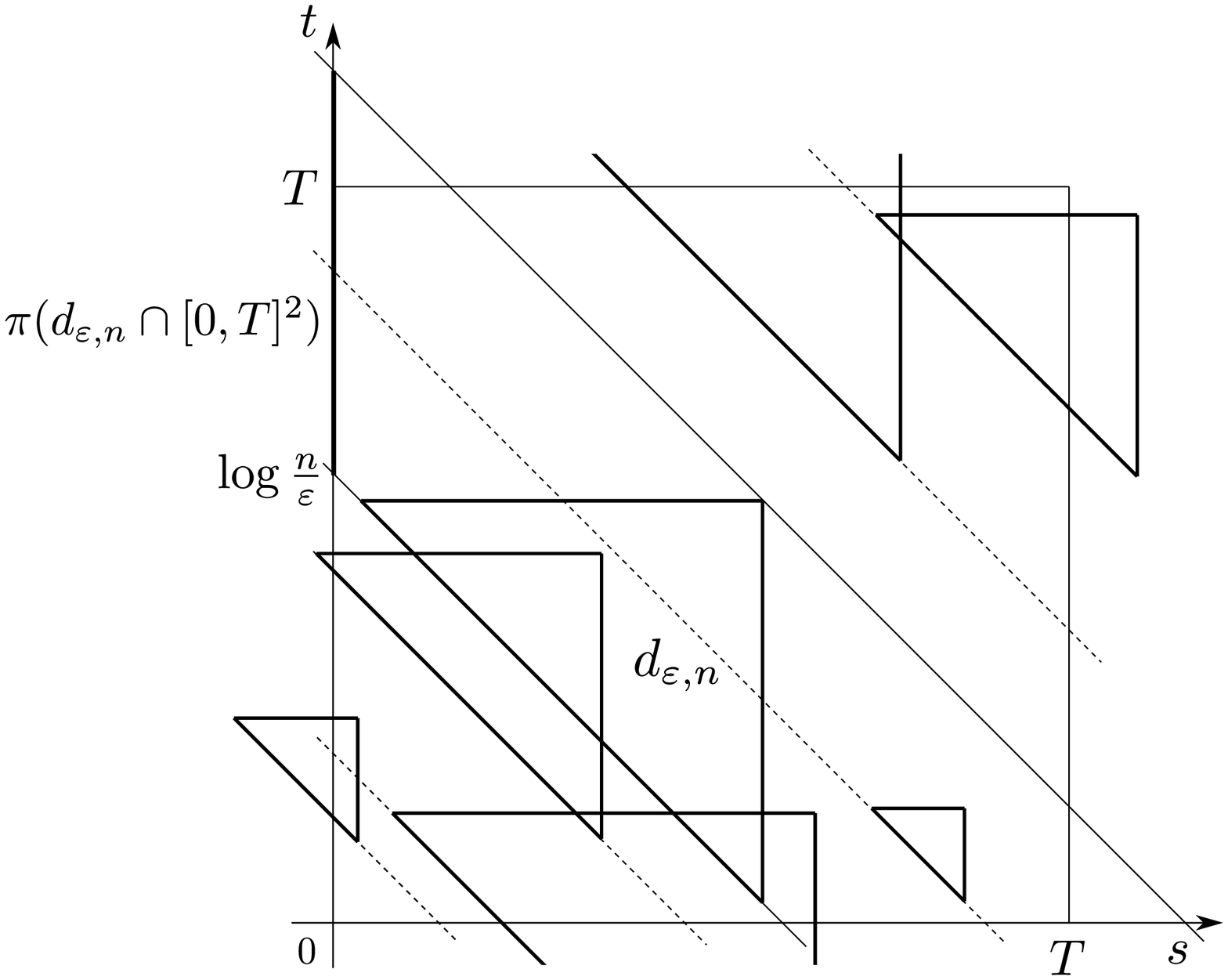}
		\caption{Intersections of $d_{\varepsilon,n} (n\in\N)$ with $[0,T]^2$ and its image under $\pi$}\label{cover_dn}
	\end{figure}
	Let $\Xi$ be the set of $n\in\N$ such that $\pi\left(d_{\varepsilon,n}\cap[0,T]^2\right)$ is an interval with positive length and maximal with respect to inclusion in $\left\{\pi\left(d_{\varepsilon,m}\cap[0,T]^2\right)\left|\ m\in\N\right.\right\}$. Then $\Xi$ is a finite set and the following holds.
	
	\begin{lem}\label{Xi_max}
		$$
		m_{\R}\left(\bigcup_{n\in\Xi}\pi\left(d_{\varepsilon,n}\cap[0,T]^2\right)\right)\geq \gamma\cdot T.
		$$
	\end{lem}
	
	\begin{proof}
		Similarly as the proof of Theorem \ref{converge_high_entrpy}, it follows from the inequality (\ref{escape_to_cusp}) and (\ref{cover_by_dn}) that
		\begin{align*}
			\gamma\cdot T^2\leq&\ m_{\R^2}\left(\left\{(s,t)\in[0,T]^2\left|\ a_{s,t}\tau_{\alpha,\beta}\Gamma\in X_{\varepsilon}\right.\right\}\right)\\
			\leq&\ T\cdot m_\R\left(\pi\left(\left\{(s,t)\in[0,T]^2\left|\ a_{s,t}\tau_{\alpha,\beta}\Gamma\in X_{\varepsilon}\right.\right\}\right)\right)\\
			=&\ T\cdot m_{\R}\left(\bigcup_{n\in\N}\pi\left(d_{\varepsilon,n}\cap[0,T]^2\right)\right)
		\end{align*}
		and hence
		\begin{equation*}
			m_{\R}\left(\bigcup_{n\in\N}\pi\left(d_{\varepsilon,n}\cap[0,T]^2\right)\right)\geq \gamma\cdot T.
		\end{equation*}
		Moreover, by the definition of $\Xi$, it is easily seen that
		\begin{equation*}
			m_\R\left(\bigcup_{n\in \Xi}\pi\left(d_{\varepsilon,n}\cap[0,T]^2\right)\right)=m_\R\left(\bigcup_{n\in\N}\pi\left(d_{\varepsilon,n}\cap[0,T]^2\right)\right)
		\end{equation*}
		and we obtain the lemma.
	\end{proof}
	
	For each $n\in \Xi$, we write $\lambda_n>0$ for the length of $\pi\left(d_{\varepsilon,n}\cap[0,T]^2\right)=[\log(n/\varepsilon),s_n+t_n]$ and 
	$$
	\Lambda(n)=\left\{2^pn\left|\ p\in\Z,0\leq p\log2\leq \lambda_n/3\right.\right\}.
	$$
	Let $2^pn\in\Lambda(n)$. Since $(s_n,t_n)\in d_{\varepsilon,n}\cap[0,T]^2$, we have $s_n+t_n\leq 2T$, $\langle n\alpha\rangle\leq \varepsilon e^{-s_n}$ and $\langle n\beta\rangle\leq \varepsilon e^{-t_n}$. Hence, we obtain that
	$$
	2^pn\leq e^{\lambda_n/3}n<e^{s_n+t_n-\log n+\log\varepsilon}n<e^{2T}
	$$
	and
	\begin{align*}
		2^pn\langle 2^pn\alpha\rangle\langle 2^pn\beta\rangle\leq (2^p)^3n\langle n\alpha\rangle\langle n\beta\rangle\leq e^{\lambda_n}\cdot n\cdot\varepsilon e^{-s_n}\cdot\varepsilon e^{-t_n}=\varepsilon^3.
	\end{align*}
	Therefore, we have
	\begin{equation}\label{inclusion_lambda}
		\bigcup_{n\in \Xi}\Lambda(n)\subset
		\left\{n\in \N\left|\ n<e^{2T}, n\langle n\alpha\rangle\langle n\beta\rangle\leq\varepsilon^3\right.\right\}.
	\end{equation}
	
	We estimate the cardinality of $\bigcup_{n\in \Xi}\Lambda(n)$. For $n,n'\in \Xi$, we write $n\sim n'$ if $\Lambda(n)\cap\Lambda(n_1)\neq\emptyset,\ \Lambda(n_1)\cap\Lambda(n_2)\neq\emptyset,\dots,\Lambda(n_i)\cap\Lambda(n')\neq\emptyset$ for some finite $n_1,n_2,\dots,n_i\in \Xi$. This $\sim$ is an equivalence relation on $\Xi$. We partition $\Xi$ into the equivalent classes under $\sim$ and write $\Xi=\bigsqcup_{l=1}^L\Xi_l$. Then we can write $\bigcup_{n\in\Xi}\Lambda(n)=\bigsqcup_{l=1}^L\bigcup_{n\in\Xi_l}\Lambda(n)$ as the disjoint union. Therefore, it is sufficient to estimate the cardinality of $\bigcup_{n\in\Xi_l}\Lambda(n)$ for each $l$. Then we see the following lemma.
	
	\begin{lem}\label{card_Lambda_n}
		For each $l$, we have
		\begin{equation*}\label{estimate_Xil_lambda}
			\left|\bigcup_{n\in\Xi_l}\Lambda(n)\right|\geq\frac{1}{3\log2}m_\R\left(\bigcup_{n\in\Xi_l}\pi(d_{\varepsilon,n}\cap[0,T]^2)\right).
		\end{equation*}
	\end{lem}
	
	If Lemma \ref{card_Lambda_n} is established, then, by Lemma \ref{card_Lambda_n}, (\ref{inclusion_lambda}) and Lemma \ref{Xi_max}, we have
	\begin{align*}
		\left|\left\{n\in\N\left|\ n<e^{2T}, n\langle n\alpha\rangle\langle n\beta\rangle\leq\varepsilon^3\right.\right\}\right|
		\geq&\ \left|\bigcup_{n\in\Xi}\Lambda(n)\right|\\
		=&\ \sum_{l=1}^L\left|\bigcup_{n\in\Xi_l}\Lambda(n)\right|\\
		\geq&\ \frac{1}{3\log2}\sum_{l=1}^Lm_\R\left(\bigcup_{n\in\Xi_l}\pi(d_{\varepsilon,n}\cap[0,T]^2)\right)\\
		\geq&\ \frac{1}{3\log2}m_\R\left(\bigcup_{n\in\Xi}\pi(d_{\varepsilon,n}\cap[0,T]^2)\right)\\
		\geq&\ \frac{\gamma}{3\log2}T
	\end{align*}
	and complete the proof of Proposition \ref{cusp_littlewood}. Hence we prove Lemma \ref{card_Lambda_n} below.
	
	\begin{proof}[Proof of Lemma \ref{card_Lambda_n}]
		We write $\Xi_l=\{n_1<n_2<\cdots<n_K\}$. We notice that, if $\Lambda(n)\cap\Lambda(n')\neq\emptyset$, then $2^pn=2^{p'}n'$ for some $p,p'\in\Z_{\geq0}$. Hence, we can write 
		\begin{equation}\label{differ_by_2}
			n_2=2^{q_2}n_1,\dots, n_K=2^{q_K}n_1
		\end{equation}
		for some $q_2<\cdots<q_K\in\N$. Since $\pi(d_{\varepsilon,n}\cap[0,T]^2)$ is maximal with respect to inclusion for each $n\in\Xi$, we have $\log(n_i/\varepsilon)<\log(n_{i+1}/\varepsilon)$ and $\log(n_i/\varepsilon)+\lambda_{n_i}<\log(n_{i+1}/\varepsilon)+\lambda_{n_{i+1}}$ for $1\leq i< K$. Moreover, if we write $2^{p_i}n_i$ for the maximal element of $\Lambda(n_i)$, then we can see that
		\begin{equation}\label{intersection_Lambda}
			\log\frac{n_{i+1}}{\varepsilon}\leq\log\frac{2^{p_i}n_i}{\varepsilon}\leq\log\frac{2^{p_{i+1}}n_{i+1}}{\varepsilon}
		\end{equation}
		(see Figure \ref{relposition}).
		\begin{figure}[b]
			\centering
			\includegraphics[width=14cm]{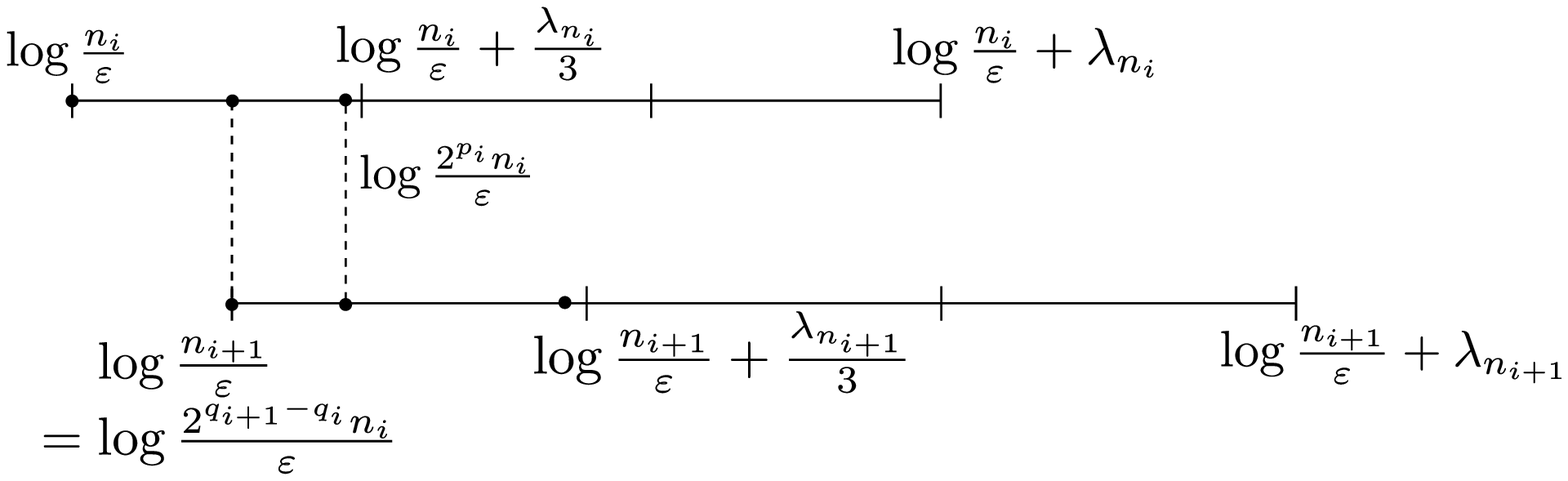}
			\caption{$\pi(d_{\varepsilon,n_i}\cap[0,T]^2)$ and $\pi(d_{\varepsilon,n_{i+1}}\cap[0,T]^2)$}\label{relposition}
		\end{figure}
		Indeed, if $\log(2^{p_i}n_i/\varepsilon)>\log(2^{p_{i+1}}n_{i+1}/\varepsilon)$, then $2^{p_{i+1}}n_{i+1}<2^{p_i}n_i$ and, since these two numbers differ only by the factor $2$ by (\ref{differ_by_2}), we have  $2^{p_{i+1}+1}n_{i+1}\leq2^{p_i}n_i$. This implies that $\log(n_{i+1}/\varepsilon)+\lambda_{n_{i+1}}/3<\log(n_i/\varepsilon)+\lambda_{n_i}/3$ and hence
		$\pi\left(d_{\varepsilon,n_{i+1}}\cap[0,T]^2\right)\subsetneq\pi(d_{\varepsilon,n_i}\cap[0,T]^2)$. This contradicts that $\pi\left(d_{\varepsilon,n_{i+1}}\cap[0,T]^2\right)$ is maximal. Hence, we obtain that
		$\log(2^{p_i}n_i/\varepsilon)\leq\log(2^{p_{i+1}}n_{i+1}/\varepsilon)$. If $\log(n_{i+1}/\varepsilon)>\log(2^{p_i}n_i/\varepsilon)$, then $2^{p_i}n_i<n_{i+1}\leq n_j$ for all $j>i$. On the other hand, as we have seen above, we have $2^{p_{j'}}n_{j'}\leq 2^{p_i}n_i$ for all $j'\leq i$. Hence, it follows that, for all $j'\leq i<j$, $\Lambda(n_{j'})\cap\Lambda(n_j)=\emptyset$ and this contradicts that $n_i\sim n_{i+1}$. Therefore, we obtain that $\log(n_{i+1}/\varepsilon)\leq\log(2^{p_i}n_i/\varepsilon)$. 
		
		For $n\in\Xi$, we write $I_n=[\log(n/\varepsilon),\log(n/\varepsilon)+\lambda_n/3]$ for the closed interval with the same left end point as $\pi(d_{\varepsilon,n}\cap[0,T]^2)$ and the length of one-third of $\pi(d_{\varepsilon,n}\cap[0,T]^2)$.
		Then, from the inequality (\ref{intersection_Lambda}), it follows that $\log(n_{i+1}/\varepsilon)\in I_{n_i}$ for $1\leq i<K$, and hence
		$I=\bigcup_{i=1}^KI_{n_i}$ is the close interval $[\log(n_1/\varepsilon),\log(n_K/\varepsilon)+\lambda_{n_K}/3]$.
		Furthermore, by the inequality (\ref{intersection_Lambda}) and (\ref{differ_by_2}), if we write $|I|$ for the length of the interval $I$, then we have
		\begin{equation*}
			\bigcup_{n\in\Xi_l}\Lambda(n)=\bigcup_{i=1}^K\Lambda(n_i)=\left\{2^pn_1\left|\ p\in\Z, 0\leq p\log2\leq|I|\right.\right\}.
		\end{equation*}
		Hence we have
		\begin{equation*}
			\left|\bigcup_{n\in\Xi_l}\Lambda(n)\right|\geq\frac{|I|}{\log2}.
		\end{equation*}
		Furthermore, for
		$$
		I=[\log(n_1/\varepsilon),\log(n_K/\varepsilon)+\lambda_{n_K}/3]\quad{\rm and}\quad \bigcup_{i=1}^K\pi(d_{\varepsilon,n_i}\cap[0,T]^2)=[\log(n_1/\varepsilon),\log(n_K/\varepsilon)+\lambda_{n_K}],
		$$
		we have
		\begin{equation*}
			|I|\geq \frac{1}{3}m_\R\left(\bigcup_{i=1}^K\pi(d_{\varepsilon,n_i}\cap[0,T]^2)\right)=
			\frac{1}{3}m_\R\left(\bigcup_{n\in\Xi_l}\pi(d_{\varepsilon,n}\cap[0,T]^2)\right).
		\end{equation*}
		These two inequalities imply that
		\begin{equation*}\
			\left|\bigcup_{n\in\Xi_l}\Lambda(n)\right|\geq\frac{1}{3\log2}m_\R\left(\bigcup_{n\in\Xi_l}\pi(d_{\varepsilon,n}\cap[0,T]^2)\right).
		\end{equation*}
	\end{proof}
	
	\section{Proof of Theorem \ref{diagonalempirical}}\label{section_proof_crutialthm}
	
	In this section, we prove Theorem \ref{diagonalempirical}.
	As we said in Subsection \ref{subsec_empirical_diagonal}, Theorem \ref{diagonalempirical} is the analogy of \cite[Theorem 1.5]{Usu22} for the $\times a,\times b$ action on $\R/\Z$
	to the diagonal action on $X=\SL(3,\R)/\SL(3,\Z)$. However, unlike $\R/\Z$, $X$ is noncompact. Hence, we need some works for the proof which are not in \cite{Usu22}.
	
	We fix $x_0\in X$ and $0<\gamma<1$ and prove Theorem \ref{diagonalempirical} by showing that $\dim_HZ_{x_0}(\gamma^2)\leq15\gamma$.
	In the following argument, we often extend the diagonal action on $X$ to the continuous action of $A$ on $\widetilde{X}=X\sqcup\{\infty\}$ with a fixed point $\infty$ and consider a finite Borel measure $\mu$ on $X$ with $\mu(X)\leq1$ as a Borel probability measure on $\widetilde{X}$ by letting $\mu(\{\infty\})=1-\mu(X)$. As we saw in Subsection \ref{subsec_high_entropy}, if $\mu$ is $a_1$-invariant, then we have $h_\mu(a_1)=\mu(X)h_{\widehat{\mu}}(a_1)$, where the left hand side is the measure-theoretic entropy of the action of $a_1$ on $\widetilde{X}$ with respect to $\mu$ and, as above, $\widehat{\mu}$ in the right hand side is the Borel probability measure on $X$ obtained by normalizing $\mu$.
	
	First, we reduce the problem of the $\R^2$-action to the $\Z^2$-action. For $x\in X$ and $N\in\N$, we define the Borel probability measure $\delta_{a_1,a_2,x}^N$ on $X$ by
	$$
	\delta_{a_1,a_2,x}^N=\frac{1}{N^2}\sum_{m,n=0}^{N-1}\delta_{a_1^ma_2^nx}
	=\frac{1}{N^2}\sum_{m,n=0}^{N-1}\delta_{a_{m,n}x}
	$$
	and call it the $N$-empirical measure of $x$ with respect to the action of $a_1$ and $a_2$.
	For a finite Borel measure $\mu$ on $X$, we say that $\delta_{a_1,a_2,x}^N$ accumulates to $\mu$ if there exists a divergent subsequence $\{N_k\}_{k=1}^\infty\subset\N$ such that $\{\delta_{a_1,a_2,x}^{N_k}\}_{k=1}^\infty$ converges to $\mu$. Then, it can be seen that $\mu$ is $a_1$ and $a_2$-invariant.
	
	\begin{lem}\label{R2toZ2}
		Let $Z'_{x_0}(\gamma^2)$ be the set of $u\in\overline{B^U_1}$ such that $\delta_{a_1,a_2,ux_0}^N$ accumulates to some $a_1$ and $a_2$-invariant finite Borel measure $\mu$ on $X$ such that
		$$
		1-\gamma^2<\mu(X)\leq1\quad{\it and}\quad h_{\widehat{\mu}}(a_1)\leq \gamma^2.
		$$
		Then we have $Z_{x_0}(\gamma^2)\subset Z'_{x_0}(\gamma^2)$.
	\end{lem}
	
	\begin{proof}
		Let $u\in Z_{x_0}(\gamma^2)$. Then there exist a sequence $\{T_k\}_{k=1}^\infty$ in $\R_{>0}$ and an $A$-invariant finite Borel measure $\mu$ on $X$ such that $T_k\to \infty$, $\delta_{A^+,ux_0}^{T_k}\to\mu$ as $k\to\infty$, $1-\gamma^2<\mu(X)\leq1$ and $h_{\widehat{\mu}}(a_1)\leq \gamma^2$. If we write $N_k=\lfloor T_k\rfloor\in\N$, then it can be seen that $\{N_k\}_{k=1}^\infty$ is a divergent subsequence of $\N$ and $\delta_{A^+,ux_0}^{N_k}\to\mu$ as $k\to\infty$. Since $\widetilde{X}$ is compact, 
		the space of Borel probability measures on $\widetilde{X}$ is compact with respect to the weak* topology and, by taking some subsequence, we can assume that $\{\delta_{a_1,a_2,ux_0}^{N_k}\}_{k=1}^\infty$ converges to some $a_1$ and $a_2$-invariant finite Borel measure $\mu'$ on $X$. For every $f\in C_0(X)$, $\int_{[0,1]^2}f(a_{s,t}x)\ dsdt\ (x\in X)$ is a continuous function on $X$ which is in $C_0(X)$. Hence, by $\delta_{a_1,a_2,ux_0}^{N_k}\to\mu'$ and $\delta_{A^+,ux_0}^{N_k}\to\mu$, we have
		\begin{align*}
			\int_{[0,1]^2}\int_Xf(a_{s,t}x)\ d\mu'(x)dsdt=&\int_X\int_{[0,1]^2}f(a_{s,t}x)\ dsdtd\mu'(x)\\
			=&\lim_{k\to\infty}\frac{1}{N_k^2}\sum_{m,n=0}^{N_k-1}\int_{[0,1]^2}f(a_{s,t}a_{m,n}ux_0)\ dsdt\\
			=&\lim_{k\to\infty}\frac{1}{N_k^2}\int_{[0,N_k]^2}f(a_{s,t}ux_0)\ dsdt\\
			=&\int_Xf\ d\mu
		\end{align*}
		for all $f\in C_0(X)$ and this implies that
		\begin{equation*}
			\mu=\int_{[0,1]^2}{a_{s,t}}_*\mu'\ dsdt.
		\end{equation*}
		From this equation, it follows that
		\begin{equation*}
			\mu(X)=\int_{[0,1]^2}\mu'(a_{s,t}^{-1}X)\ dsdt=\mu'(X).
		\end{equation*}
		Hence, we have $1-\gamma^2<\mu(X)=\mu'(X)\leq1$ and
		\begin{equation*}
			\widehat{\mu}=\int_{[0,1]^2}{a_{s,t}}_*\widehat{\mu'}\ dsdt.
		\end{equation*}
		We notice that the right hand side is the generalized convex combination of $a_1$ and $a_2$-invariant Borel probability measures on X. By considering these measures as Borel probability measures on $\widetilde{X}$, we can apply Proposition \ref{entropyrep} to this equation and, since $a_1$ and $a_{s,t}$ are commutative, we obtain that
		\begin{align*}
			h_{\widehat{\mu}}(a_1)=&\int_{[0,1]^2}h_{{a_{s,t}}_*\widehat{\mu'}}(a_1)\ dsdt\\
			=&\int_{[0,1]^2}h_{\widehat{\mu'}}(a_1)\ dsdt\\=&\ h_{\widehat{\mu'}}(a_1).
		\end{align*}
		Hence, we have $h_{\widehat{\mu'}}(a_1)=h_{\widehat{\mu}}(a_1)\leq \gamma^2$. Then we showed that $u\in Z'_{x_0}(\gamma^2)$.
	\end{proof}
	
	For a compact subset $K\subset X$, we write $Z'_{x_0}(\gamma^2,K)$ for the set of $u\in\overline{B^U_1}$ such that $\delta_{a_1,a_2,ux_0}^N$ accumulates to some $a_1$ and $a_2$-invariant finite Borel measure $\mu$ such that
	\begin{equation*}
		1-\gamma^2<\mu(\Int K)\leq\mu(X)\leq1\quad{\rm and}\quad h_{\widehat{\mu}}(a_1)\leq \gamma^2,
	\end{equation*}
	where $\Int K$ is the interior of $K$ in $X$. If we take an increasing sequence $K_1\subset K_2\subset\cdots K_n\subset\cdots$ of compact subsets of $X$ such that $X=\bigcup_{n=1}^\infty\Int K_n$, then it can be seen that
	\begin{equation*}
		Z'_{x_0}(\gamma^2)=\bigcup_{n=1}^\infty Z'_{x_0}(\gamma^2,K_n).
	\end{equation*}
	Hence, if we can show $\dim_H Z'_{x_0}(\gamma^2,K)\leq 15\gamma$ for any compact subset $K\subset X$, then we have $\dim_H Z_{x_0}(\gamma^2)\leq\dim_H Z'_{x_0}(\gamma^2)\leq 15\gamma$ and complete the proof of Theorem \ref{diagonalempirical}.
	
	We fix a compact subset $K\subset X$ and show $\dim_H Z'_{x_0}(\gamma^2,K)\leq 15\gamma$. 
	First, we define some open covers of $\widetilde{X}$ we need for the proof.
	We take a right invariant metric on $G$ and write $d_G$. We also write $B^G_\rho=\left\{g\in G\left|\ d_G(e,g)<\rho\right.\right\}$. Then, for the compact subset $K\subset X$, we can take a sufficiently small constant $0<\rho_K<1$ such that, for every $x\in K$, the map $B^G_{\rho_K}\ni g\mapsto gx\in X$ is injective (see \cite[Proposition 9.14]{EW11}). Since $U$ is a closed subgroup of $G$, we can take sufficiently small constants $0<\rho'<\rho<\rho_K$ such that $\overline{B^U_{e^2\rho}}\subset B^G_{\rho_K/2}$ and $B^G_{\rho'}\cap U\subset B^U_\rho$. We take a finite open cover $B=\{B_1,\dots,B_{k-1},B_k\}$ of $\widetilde{X}$ such that $\{B_1,\dots,B_{k-1}\}$ is a finite open cover of $K$ in $X$ such that $B_i=B^G_{\rho'/2}y_i,y_i\in K$ for $i=1,\dots, k-1$ and $B_k=\widetilde{X}\setminus K$. For this $B$, we take a Borel measurable and finite partition $\alpha=\{\alpha_1,\dots,\alpha_k\}$ of $\widetilde{X}$ such that $\alpha_i\subset B_i$ for $i=1,\dots,k$. For every $M\in\N$, we write $\alpha_M=\{\alpha_{M,1},\dots,\alpha_{M,k_M}\}=\bigvee_{m=0}^{M-1}a_1^{-m}\alpha$ and $\alpha_{M,i}=\bigcap_{m=1}^{M-1}a_1^{-m}\alpha_{j_{i,m}}\ (j_{i,m}\in\{1,\dots,k\})$ for $i=1,\dots,k_M$. Then we take a finite open cover $\beta_M=\{\beta_{M,1},\dots,\beta_{M,k_M}\}$ of $\widetilde{X}$ such that 
	\begin{equation}\label{patitionalphacoverbeta}
		\alpha_{k,i}=\bigcap_{m=1}^{M-1}a_1^{-m}\alpha_{j_{i,m}}\subset\beta_{k,i}\subset \bigcap_{m=0}^{M-1}a_1^{-m}B_{j_{m,i}}
	\end{equation}
	for $i=1,\dots,k_M$.
	
	We extract an important property of elements in $Z'_{x_0}(\gamma^2,K)$ to estimate Hausdorff dimension.
	We need some definitions.
	Let $k\in \N$. A {\it $k$-distribution} is an element $q=(q_1,\dots,q_k)$ of $\R^k$ such that $q_i\geq0\ (i=1,\dots,k)$ and $\sum_{i=1}^kq_i=1$. We write $D_k$ for the set of all $k$-distributions. For $q\in D_k$, we write $H(q)=-\sum_{i=1}^kq_i\log q_i\in\R_{\geq0}$ for {\it the entropy of $q$}. If $N\in\N$ and $c=(c_1,\dots,c_N)\in\{1,\dots,k\}^N$, we define $\dist(c)=(\dist(c)_1,\dots,\dist(c)_k)\in D_k$ by $\dist(c)_i=N^{-1}\left|\left\{n\in\{1,\dots,N\}\left|\ c_n=i\right.\right\}\right|\ (i=1,\dots,k)$. 
	Let $\beta=\{\beta_1,\dots,\beta_k\}$ be a finite cover of $\widetilde{X}$. For $x\in X$ and $N\in\N$, we say that $(\beta_{i_n})_{n=0,\dots,N-1}\in\beta\times\dots\times\beta$ is an {\it $N$-choice for $x$ with respect to $a_1$ and $\beta$} if we have $a_1^nx\in\beta_{i_n}$ for $n=0,\dots,N-1$. Then, for $(\beta_{i_n})_{n=0,\dots,N-1}$, we write $q\left((\beta_{i_n})_{n=0,\dots,N-1}\right)=\dist(i_0,\dots,i_{N-1})\in D_k$. Then we state Lemma \ref{keyinclusion} as follows.
	
	\begin{lem}\label{keyinclusion}
		Every $u\in Z'_{x_0}(\gamma^2,K)$ satisfies the following condition. For any $0<\varepsilon<1$ and $M_0\in\N$, there exists $M\in\N$ such that $M\geq M_0$ and, for infinitely many $N\in\N$, the following holds: there exists $n\in\Z$ such that $0\leq n<\gamma N$,
		$$
		\frac{1}{M}H(q((\beta_{M,i_m})_{m=0,\dots,N-1}))<2\gamma+\varepsilon
		$$
		for some $N$-choice $(\beta_{M,i_m})_{m=0,\dots,N-1}$ for $a_2^nux_0$ with respect to $a_1$ and $\beta_M$, and
		$$
		\frac{1}{N}\left|\left\{m\in\{0,1\dots,N-1\}\left|\ a_1^ma_2^nux_0\in X\setminus\Int K\right.\right\}\right|<2\gamma+\varepsilon.
		$$
	\end{lem}
	
	We notice that, in Lemma \ref{keyinclusion}, we consider each $\beta_M$ as an open cover of $X$ by restriction. We write $Q_{x_0}(\gamma^2,K,\{\beta_M\}_{M=1}^\infty)$ for the set of $u\in\overline{B^U_1}$ which satisfies the condition stated in Lemma \ref{keyinclusion}. Then we have $\dim_H Z'_{x_0}(\gamma^2,K)\leq\dim_H Q_{x_0}(\gamma^2,K,\{\beta_M\}_{M=1}^\infty)$ and we will  estimate the right hand side later.
	
	For the proof of Lemma \ref{keyinclusion}, we need the one more notion.
	For $x\in X$ and $N\in\N$, we say that $(\beta_{i_{m,n}})_{m,n=0,\dots,N-1}\in\beta^{\{(m,n)\in \Z^2\left|\ 0\leq m,n<N\right.\}}$ is an {\it $N$-choice for $x$ with respect to $a_1,a_2$ and $\beta$} if we have $a_1^ma_2^nx=a_{m,n}x\in \beta_{i_{m,n}}$ for $0\leq m,n<N$. Then, for $(\beta_{i_{m,n}})_{m,n=0,\dots,N-1}$, we write $q((\beta_{i_{m,n}})_{m,n=0,\dots,N-1})=\dist((i_{m,n})_{m,n=0,\dots,N-1})\in D_k$. We notice that, if we fix $n=0,\dots,N-1$, then $(\beta_{i_{m,n}})_{m=0,\dots,N-1}\in\beta\times\cdots\times\beta$ is an $N$-choice for $a_2^nx$ with respect to $a_1$ and $\beta$ and we have $q((\beta_{i_{m,n}})_{m,n=0,\dots,N-1})=N^{-1}\sum_{n=0}^{N-1}q((\beta_{i_{m,n}})_{m=0,\dots,N-1})\in D_k$.
	
	\begin{proof}[Proof of Lemma \ref{keyinclusion}]
		Let $u\in Z'_{x_0}(\gamma^2,K)$. Then there exist a divergent subsequence $\{N_j\}_{j=1}^\infty$ of $\N$ and an $a_1$ and $a_2$-invariant finite Borel measure $\mu$ on $X$ such that
		\begin{enumerate}
			\renewcommand{\labelenumi}{\rm{(\roman{enumi})}}
			\item $\delta^{N_j}_{a_1,a_2,ux_0}\to\mu$ as $j\to\infty$,
			\item $1-\gamma^2<\mu(\Int K)\leq\mu(X)\leq1$, and
			\item $h_{\widehat{\mu}}(a_1)\leq \gamma^2$.
		\end{enumerate}
		We consider $\mu$ as the $a_1$ and $a_2$-invariant Borel probability measure on $\widetilde{X}$ such that $\mu(\{\infty\})=1-\mu(X)$. Then, as we said above, we have $h_\mu(a_1)=\mu(X)h_{\widehat{\mu}}(a_1)\leq \gamma^2$ by (iii). 
		
		We take arbitrary $0<\varepsilon<1$.
		Since $h_\mu(a_1,\alpha)=\lim_{M\to\infty}M^{-1}H_\mu(\alpha_M)\leq h_\mu(a_1)<\gamma^2+\gamma\varepsilon$, we have
		\begin{equation*}
			\frac{1}{M}H_\mu(\alpha_M)<\gamma^2+\gamma\varepsilon
		\end{equation*}
		for sufficiently large $M\in\N$, where $H_\mu(\alpha_M)=H(\mu(\alpha_{M,1}),\dots,\mu(\alpha_{M,k_M}))$ for the $k_M$-distribution $(\mu(\alpha_{M,1}),\dots,\mu(\alpha_{M,k_M}))$. We fix such $M\in\N$. Since $H(q)\ (q\in D_{k_M})$ is continuous, there exists $0<\eta<1$ such that, for $q\in D_{k_M}$,
		\begin{equation}\label{approxentropy}
			\left|q-(\mu(\alpha_{M,1}),\dots,\mu(\alpha_{M,k_M}))\right|<\eta\quad{\rm implies}\quad\frac{1}{M}H(q)<\gamma^2+\gamma\varepsilon,
		\end{equation}
		where $|\cdot|$ denotes the Euclidean norm on $\R^{k_M}$. By the regularity of the Borel probability measure $\mu$ on $\widetilde{X}$, for each $i=1,\dots,k_M$, there exists a compact subset $C_i\subset \widetilde{X}$ such that $C_i\subset\alpha_{M,i}$ and $\mu(\alpha_{M,i}\setminus C_i)<\eta/2\sqrt{k_M}k_M$. Since $C_1,\dots,C_{k_M}$ are compact, pairwise disjoint and $C_i\subset\alpha_{M,i}\subset\beta_{M,i}$, we can take an open subset $V_i\subset\widetilde{X}$ for each $i=1,\dots,k_M$ such that $C_i\subset V_i\subset \beta_{M,i}$ and $V_1,\dots, V_{k_M}$ are pairwise disjoint. By (i), if $j$ is sufficiently large, we have
		\begin{equation*}
			\delta^{N_j}_{a_1,a_2,ux_0}(V_i)>\mu(V_i)-\frac{\eta}{2\sqrt{k_M}k_M}\geq \mu(C_i)-\frac{\eta}{2\sqrt{k_M}k_M}
		\end{equation*}
		and hence
		\begin{equation}\label{accumulatemu}
			\delta^{N_j}_{a_1,a_2,ux_0}(V_i)>\mu(\alpha_{M,i})-\frac{\eta}{\sqrt{k_M}k_M}
		\end{equation}
		for $i=1,\dots, k_M$. Since $V_1,\dots,V_{k_M}$ are pairwise disjoint and $V_i\subset\beta_{M,i}$, for sufficiently large $j$, we can take an $N_j$-choice $(\beta_{M,i_{m,n}})_{m,n=0,\dots,N_j-1}$ for $ux_0$ with respect to $a_1,a_2$ and $\beta_M$ such that $i_{m,n}=i$ whenever $a_1^ma_2^nux_0\in V_i$. We write $q((\beta_{M,i_{m,n}})_{m,n=0,\dots,N_j-1})=(q_1,\dots,q_{k_M})$. Then, by the inequality (\ref{accumulatemu}), we have
		\begin{equation*}
			q_i\geq\delta^{N_j}_{a_1,a_2,ux_0}(V_i)>\mu(\alpha_{M,i})-\frac{\eta}{\sqrt{k_M}k_M}
		\end{equation*}
		for $i=1,\dots,k_M$. Since $(\mu(\alpha_{M,1}),\dots,\mu(\alpha_{M,k_M}))$ and $q((\beta_{M,i_{m,n}})_{m,n=0,\dots,N_j-1})$ are both $k_M$-distribution, from the above inequality, we have
		\begin{equation*}
			\left|q((\beta_{M,i_{m,n}})_{m,n=0,\dots,N_j-1})-(\mu(\alpha_{M,1}),\dots,\mu(\alpha_{M,k_M}))\right|<\eta
		\end{equation*}
		and hence, by (\ref{approxentropy}),
		\begin{equation}\label{entropyofchoice}
			\frac{1}{M}H\left(q((\beta_{M,i_{m,n}})_{m,n=0,\dots,N_j-1})\right)<\gamma^2+\gamma\varepsilon.
		\end{equation}
		Now, as we said above, we have the convex combination
		\begin{align*}
			&q((\beta_{M,i_{m,n}})_{m,n=0\dots,N_j-1})\\=&\ \frac{1}{N_j}\sum_{n=0}^{N_j-1}q((\beta_{M,i_{m,n}})_{m=0,\dots,N_j-1})\\
			=&\ \frac{\lfloor \gamma N_j\rfloor+1}{N_j}\frac{1}{\lfloor \gamma N_j\rfloor+1}\sum_{0\leq n<\gamma N_j}q((\beta_{M,i_{m,n}})_{m=0,\dots,N_j-1})\\&\qquad\qquad+\frac{N_j-\lfloor \gamma N_j\rfloor-1}{N_j}\frac{1}{N_j-\lfloor \gamma N_j\rfloor-1}\sum_{\gamma N_j\leq n<N_j}q((\beta_{M,i_{m,n}})_{m=0,\dots,N_j-1})
		\end{align*}
		of elements in $D_{k_M}$
		\footnote{Here we assume that $tN_j$ is not an integer. Even if $tN_j$ is an integer, we can do the same argument by replacing $\lfloor tN_j\rfloor+1$ to $tN_j$.}
		. Since $H(q)\ (q\in D_{k_M})$ is concave, we have
		\begin{align*}
			&H\left(q((\beta_{M,i_{m,n}})_{m,n=0,\dots,N_j-1})\right)\\
			\geq&\ \frac{\lfloor \gamma N_j\rfloor+1}{N_j}H\left(\frac{1}{\lfloor \gamma N_j\rfloor+1}\sum_{0\leq n<\gamma N_j}q((\beta_{M,i_{m,n}})_{m=0,\dots,N_j-1})\right)\\
			&\qquad\qquad+\frac{N_j-\lfloor \gamma N_j\rfloor-1}{N_j}H\left(\frac{1}{N_j-\lfloor \gamma N_j\rfloor-1}\sum_{\gamma N_j\leq n<N_j}q((\beta_{M,i_{m,n}})_{m=0,\dots,N_j-1})\right)\\
			\geq&\ \gamma H\left(\frac{1}{\lfloor \gamma N_j\rfloor+1}\sum_{0\leq n<\gamma N_j}q((\beta_{M,i_{m,n}})_{m=0,\dots,N_j-1})\right)\\
			\geq&\ \frac{\gamma}{\lfloor \gamma N_j\rfloor+1}\sum_{0\leq n<\gamma N_j}H\left(q((\beta_{M,i_{m,n}})_{m=0,\dots,N_j-1})\right)
		\end{align*}
		and hence, by the inequality (\ref{entropyofchoice}),
		\begin{equation}\label{t_entropyofchoice}
			\frac{1}{\lfloor \gamma N_j\rfloor+1}\sum_{0\leq n<\gamma N_j}\frac{1}{M}H\left(q((\beta_{M,i_{m,n}})_{m=0,\dots,N_j-1})\right)<\gamma+\varepsilon.
		\end{equation}
		
		Here, by (ii), we have $\mu(\widetilde{X}\setminus\Int K)<\gamma^2$. Since $\widetilde{X}\setminus \Int K\subset \widetilde{X}$ is compact, from (i), it follows that
		\begin{equation}\label{Z2escapeofmass}
			\delta^{N_j}_{a_1,a_2,ux_0}(\widetilde{X}\setminus\Int K)=\frac{1}{N_j^2}\left|\left\{(m,n)\in\{0,\dots,N_j-1\}^2\left|\ a_1^ma_2^nux_0\in X\setminus\Int K\right.\right\}\right|<\gamma^2.
		\end{equation}
		for sufficiently large $j$.
		Furthermore, we have
		\begin{align*}
			&\frac{1}{N_j^2}\left|\left\{(m,n)\in\{0,\dots,N_j-1\}^2\left|\ a_1^ma_2^nux_0\in X\setminus\Int K\right.\right\}\right|\\
			=&\ \frac{1}{N_j}\sum_{n=0}^{N_j-1}\frac{1}{N_j}\left|\left\{m\in\{0,\dots,N_j-1\}\left|\ a_1^ma_2^nux_0\in X\setminus\Int K\right.\right\}\right|\\
			\geq&\ \frac{\lfloor \gamma N_j\rfloor+1}{N_j}\frac{1}{\lfloor \gamma N_j\rfloor+1}\sum_{0\leq n<\gamma N_j}\frac{1}{N_j}\left|\left\{m\in\{0,\dots,N_j-1\}\left|\ a_1^ma_2^nux_0\in X\setminus\Int K\right.\right\}\right|\\
			\geq&\ \frac{\gamma}{\lfloor \gamma N_j\rfloor+1}\sum_{0\leq n<\gamma N_j}\frac{1}{N_j}\left|\left\{m\in\{0,\dots,N_j-1\}\left|\ a_1^ma_2^nux_0\in X\setminus\Int K\right.\right\}\right|
		\end{align*}
		and hence, by (\ref{Z2escapeofmass}),
		\begin{equation}\label{t_escapeofmass}
			\frac{1}{\lfloor \gamma N_j\rfloor+1}\sum_{0\leq n<\gamma N_j}\frac{1}{N_j}\left|\left\{m\in\{0,\dots,N_j-1\}\left|\ a_1^ma_2^nux_0\in X\setminus\Int K\right.\right\}\right|<\gamma.
		\end{equation}
		By adding (\ref{t_entropyofchoice}) and (\ref{t_escapeofmass}), we can see that, for sufficiently large $j$, there exists $0\leq n<\gamma N_j$ such that
		\begin{equation*}
			\frac{1}{M}H\left(q((\beta_{M,i_{m,n}})_{m=0,\dots,N_j-1})\right)+\frac{1}{N_j}\left|\left\{m\in\{0,\dots,N_j-1\}\left|\ a_1^ma_2^nux_0\in X\setminus\Int K\right.\right\}\right|<2\gamma+\varepsilon.
		\end{equation*}
		As we saw above, $(\beta_{M,i_{m,n}})_{m=0,\dots,N_j-1}$ is an $N_j$-choice for $a_2^nux_0$ with respect to $a_1$ and $\beta_M$. We remember that the above $N_j$ for sufficiently large $j$ for $M$ and $\varepsilon$ is arbitrary, and sufficiently large $M$ for $\varepsilon$ is arbitrary. Since we took $0<\varepsilon<1$ arbitrarily, we showed that $u\in Z'_{x_0}(\gamma^2, K)$ satisfies the condition in the lemma.
	\end{proof}
	
	As we said, we write $Q_{x_0}(\gamma^2,K,\{\beta_M\}_{M=1}^\infty)$ for the set of $u\in \overline{B^U_1}$ which satisfies the condition in Lemma \ref{keyinclusion}, that is,
	for any $0<\varepsilon<1$ and $M_0\in\N$, there exists $M\in\N$ such that $M\geq M_0$ and, for infinitely many $N\in\N$, the following holds: there exists $n\in\Z$ such that $0\leq n<\gamma N$,
	$M^{-1}H(q((\beta_{M,i_m})_{m=0,\dots,N-1}))<2\gamma+\varepsilon$
	for some $N$-choice $(\beta_{M,i_m})_{m=0,\dots,N-1}$ for $a_2^nux_0$ with respect to $a_1$ and $\beta_M$, and $N^{-1}\left|\left\{m\in\{0,1\dots,N-1\}\left|\ a_1^ma_2^nux_0\in X\setminus\Int K\right.\right\}\right|<2\gamma+\varepsilon$.
	The only thing which we have to do is to prove the following proposition.
	
	\begin{prop}\label{keystatement}
		$$
		\dim_H Q_{x_0}(\gamma^2,K,\{\beta_M\}_{M=1}^\infty)\leq 15\gamma.
		$$
	\end{prop}
	
	For the proof, we need the following two lemmas. These lemmas are essential for our estimate of Hausdorff dimension. The first one is due to R. Bowen. See \cite{Bow73} for the proof.
	
	\begin{lem}[{\cite[Lemma 4]{Bow73}}]\label{countchoice}
		For $k,N\in\N$ and $t\geq 0$, we write
		$$
		R(k,N,t)=\left\{c\in\{1,\dots,k\}^N\left|H(\dist(c))\leq t\right.\right\}.
		$$
		Then, fixing $k$ and $t$, we have
		$$
		\limsup_{N\to\infty}\frac{1}{N}\log\left|R(k,N,t)\right|\leq t.
		$$
	\end{lem}
	
	Before we state the second lemma, we remember the definition of the finite open cover $B=\{B_1,\dots,B_{k-1},B_k\}$ of $\widetilde{X}$ which we took when we defined $\alpha$ and $\beta_M$. We regard $B$ as the finite open cover of $X$ by restriction. Then $\{B_1,\dots,B_{k-1}\}$ is a finite open cover of $K$ in $X$ such that $B_i=B^G_{\rho'/2}y_i,y_i\in K$ for $i=1,\dots, k-1$, where $0<\rho'<\rho<\rho_K$ are the small constants we took for $K$, and $B_k=X\setminus K$. For a subset $E$ of $U$, we write $\diam_1 E$ and $\diam_2 E$ for
	\begin{align*}
		\diam_1E=&\sup\left\{|v_1-v'_1|\left|
		v=
		\begin{pmatrix}
			1&0&0\\v_1&1&0\\v_2&0&1
		\end{pmatrix},
		v'=
		\begin{pmatrix}
			1&0&0\\v'_1&1&0\\v'_2&0&1
		\end{pmatrix}\in E
		\right.\right\},\\
		\diam_2E=&\sup\left\{|v_2-v'_2|\left|
		v=
		\begin{pmatrix}
			1&0&0\\v_1&1&0\\v_2&0&1
		\end{pmatrix},
		v'=
		\begin{pmatrix}
			1&0&0\\v'_1&1&0\\v'_2&0&1
		\end{pmatrix}\in E
		\right.\right\}.
	\end{align*}
	
	\begin{lem}\label{partition_B}
		Let $N\in\N$ and $E$ be a subset of $U$ with $\diam E\leq\rho$ satisfying the following condition.
		There exist $x\in X$ and $j_0,\dots,j_{N-1}\in\{1,\dots,k-1,k\}$ such that, for any $v\in E$,
		\begin{equation}\label{condition_B}
			vx\in\bigcap_{m=0}^{N-1}a_1^{-m}B_{j_m}.
		\end{equation}
		If we write $N'=\left|\left\{m\in\{0,\dots,N-1\}\left|\ j_m=k\right.\right\}\right|$,
		then $E$ is partitioned into at most $(4e^3)^{N'}$ subsets so that each $E'$ of the subsets satisfies $\diam_1 E'\leq \rho e^{-2(N-1)}$ and $\diam_2 E'\leq \rho e^{-(N-1)}$.
	\end{lem}
	
	\begin{proof}
		We divide the interval $[0,N-1]$ of integers into two family of subintervals $\{I_p\}_{p=1}^P$ and $\{J_q\}_{q=1}^Q$:
		$$
		[0,N-1]=\bigsqcup_{p=1}^P I_p\sqcup\bigsqcup_{q=1}^Q J_q
		$$
		so that $\max I_p<\min I_{p+1}$, $\max J_q<\min J_{q+1}$,
		\begin{equation*}
			\bigsqcup_{p=1}^PI_p=\left\{m\in[0,N-1]\left|\ j_m\neq k\right.\right\}\quad{\rm and}\quad
			\bigsqcup_{q=1}^QJ_q=\left\{m\in[0,N-1]\left|\ j_m=k\right.\right\}.
		\end{equation*}
		For each $L\in[0,N-1]$ such that
		$$
		[0,L]=\bigsqcup_{p=1}^RI_p\sqcup\bigsqcup_{q=1}^SJ_q
		$$
		for some $1\leq R\leq P$ and $1\leq S\leq Q$, we inductively partition $E$ into at most $4^Se^{3\sum_{q=1}^S|J_q|}$ subsets so that each $E'$ of the subsets satisfies $\diam_1 E'\leq\rho e^{-2L}$ and $\diam_2 E'\leq\rho e^{-L}$.
		Here, for an interval $J$ of integers, $|J|$ denotes the number of elements in $J$.
		The base case of our induction is $[0,L]=I_1$ or $[0,L]=J_1$
		
		Suppose $[0,L]=I_1$. If $L=0$, then we have nothing to prove. We assume $L>0$ and show that
		$\diam_1 E\leq \rho e^{-2L}$ and $\diam_2 E\leq \rho e^{-L}$. We take arbitrary two elements $v,v'$ in $E$. Then we have
		$$
		v'=wv,\quad
		w=\begin{pmatrix}
			1&0&0\\w_1&1&0\\w_2&0&1
		\end{pmatrix}\in U\quad{\rm and}\quad
		|w_1|,|w_2|\leq \rho.
		$$
		Since $v,v'\in E$, by (\ref{condition_B}), we have
		$a_1vx\in B_{j_1}$ and $ a_1v'x=a_1wa_1^{-1}a_1vx\in B_{j_1}$. We write $x'=a_1vx\in B_{j_1}$ and
		$$
		w'=a_1wa_1^{-1}=
		\begin{pmatrix}
			1&0&0\\e^2w_1&1&0\\ew_2&0&1
		\end{pmatrix}\in\overline{B^U_{e^2\rho}}.
		$$
		From $1\in[0,L]=I_1$ and the definition of $\{I_p\}_{p=1}^P$, we see that $j_1\neq k$ and $B_{j_1}$ has the form
		$B_{j_1}=B^G_{\rho'/2}y_{j_1}$ for some $y_{j_1}\in K$. Since $x', w'x'\in B_{j_1}$, we have
		$x'=gy_{j_1}, w'x'=g'y_{j_1}$ for some $g,g'\in B^G_{\rho'/2}$ and hence $w'x'=g'y_{j_1}=w'gy_{j_1}$. Furthermore, by the definition of $0<\rho'<\rho<\rho_K$, we have $d_G(e,g')<\rho'/2<\rho_K$ and
		\begin{equation*}
			d_G(e,w'g)=d_G(g^{-1},w')\leq d_G(g^{-1},e)+d_G(e,w')<\rho'/2+\rho_K/2<\rho_K.
		\end{equation*}
		Since $\rho_K$ is an injectivity radius on $K$, we have $g'=w'g$ and hence
		$$
		w'=g'g^{-1}\in B^G_{\rho'}\cap U\subset B^U_\rho.
		$$
		Then we have $|e^2w_1|,|ew_2|\leq\rho$ and hence $|v_1-v'_1|=|w_1|\leq\rho e^{-2}$, $|v_2-v'_2|=|w_2|\leq\rho e^{-1}$. We show that $\diam_1 E\leq \rho e^{-2}$ and $\diam_2 E\leq \rho e^{-1}$. Since $[0,L]=I_1$, we can repeat this argument until $L$ and we obtain that $\diam_1 E\leq \rho e^{-2L}$ and $\diam_2 E \leq \rho e^{-L}$.
		
		If $[0,L]=J_1$, then we partition $E$ into subsets so that each $E'$ of the subsets satisfies
		$\diam_1 E'\leq \rho e^{-2L}$ and $\diam_2 E'\leq \rho e^{-L}$. Since $\diam_1 E,\diam_2 E\leq\rho$, we can partition $E$ so that the number of subsets which appear is bounded by $\lceil \rho (\rho e^{-2L})^{-1}\rceil\lceil \rho(\rho e^{-L})^{-1}\rceil\leq 2e^{2L}\cdot 2e^L<4e^{3|J_1|}$. Then we prove the statement at the base case.
		
		Assume that, for $0\leq L<N-1$ such that $[0,L]=\bigsqcup_{p=1}^RI_p\sqcup\bigsqcup_{q=1}^SJ_q$ for some $1\leq R\leq P$ and $1\leq S\leq Q$, $E$ is partitioned into at most $4^S e^{3\sum_{q=1}^S|J_q|}$ subsets so that each $E'$ of the subsets satisfies $\diam_1 E'\leq \rho e^{-2L}$ and $\diam_2 E'\leq \rho e^{-L}$. Then the next step of our induction is $0<L'\leq N-1$ such that
		$[0,L']=\bigsqcup_{p=1}^{R+1}I_p\sqcup\bigsqcup_{q=1}^SJ_q$ or $[0,L']=\bigsqcup_{p=1}^RI_p\sqcup\bigsqcup_{q=1}^{S+1}J_q$. Suppose $[0,L']=\bigsqcup_{p=1}^{R+1}I_p\sqcup\bigsqcup_{q=1}^SJ_q$. Then, for each $E'$ of the subsets of $E$ which appear in the partition of $E$ in the assumption of our induction, we can apply the same argument as above to $E'$ and show that $\diam_1 E'\leq \rho e^{-2(L+|I_{R+1}|)}=\rho e^{-2L'}$ and $\diam_2 E'\leq \rho e^{-L'}$. Hence, the statement also holds at this step. If $[0,L']=\bigsqcup_{p=1}^RI_p\sqcup\bigsqcup_{q=1}^{S+1}J_q$, then we partition each $E'$ into subsets so that each $E''$ of the subsets satisfies $\diam_1 E''\leq \rho e^{-2L'}$ and $\diam_2 E''\leq \rho e^{-L'}$. Since $\diam_1 E'\leq \rho e^{-2L}$ and $\diam_2 E'\leq \rho e^{-L}$, we can partition $E'$ so that the number of subsets which appear is bounded by $\lceil\rho e^{-2L}(\rho e^{-2L'})^{-1}\rceil\lceil\rho e^{-L}(\rho e^{-L'})^{-1}\rceil=\lceil e^{2|J_{S+1}|}\rceil\lceil e^{|J_{S+1}|}\rceil\leq 4e^{3|J_{S+1}|}$. Hence, the number of subsets of $E$ which appear in the partition of $E$ at this step is bounded by $4^S e^{3\sum_{q=1}^S|J_q|}\cdot 4e^{3|J_{S+1}|}=4^{S+1} e^{3\sum_{q=1}^{S+1}|J_q|}$ and we see that the statement holds at this step. We complete our induction.
		
		We let $L=N-1$ in the statement. We have
		$\sum_{q=1}^Q|J_q|=\left|\left\{m\in[0,N-1]\left|\ j_m=k\right.\right\}\right|=N'$ and $Q\leq\sum_{q=1}^Q|J_q|=N'$. Hence, we see that $E$ can be partitioned into at most $(4e^3)^{N'}$ subsets so that each $E'$ of the subsets satisfies
		$\diam_1 E'\leq\rho e^{-2(N-1)}$ and $\diam_2 E'\leq\rho e^{-(N-1)}$.
	\end{proof}
	
	We begin the proof of Proposition \ref{keystatement}.
	
	\begin{proof}[Proof of Proposition \ref{keystatement}]
		We take arbitrary $0<\varepsilon<1$ and $M_0\in\N$. By Lemma \ref{countchoice}, for every $M\in\N$ such that $M\geq M_0$, there exists $N_{\varepsilon, M}\in\N$ such that
		\begin{equation}\label{Mt_countchoice}
			\left|R(k_M,N,M(2\gamma+\varepsilon))\right|\leq e^{NM(2\gamma+2\varepsilon)}
		\end{equation}
		for all $N\geq N_{\varepsilon,M}$. Here, we can take $N_{\varepsilon,M}$ so large that it satisfies
		\begin{equation}\label{solargeN}
			M^2(k_M)^M\sum_{N\geq N_{\varepsilon,M}}Ne^{-MN\varepsilon}\leq 2^{-M}.
		\end{equation}
		
		The entropy $H(q)$ is uniformly continuous in $q\in D_{k_M}$ for all $M$. Hence, we can see that, if $u\in Q_{x_0}(\gamma^2,K,\{\beta_M\}_{M=1}^\infty)$, then there exists $M\in\N$ such that $M\geq M_0$ and, for infinitely many $N\in\N$, the following holds: there exists $n\in\Z$ such that $0\leq n<\gamma MN$,
		\begin{equation}\label{condition_of_entropy}
			\frac{1}{M}H\left(q((\beta_{M,i_m})_{m=0,\dots,MN-1})\right)<2\gamma+\varepsilon
		\end{equation}
		for some $MN$-choice $(\beta_{M,i_m})_{m=0,\dots,MN-1}$ for $a_2^nux_0$ with respect to $a_1$ and $\beta_M$, and
		\begin{equation}\label{condition_of_escape_of_mass}
			\frac{1}{MN}\left|\left\{m\in\{0,\dots,MN-1\}\left|\ a_1^ma_2^nux_0\in X\setminus\Int K\right.\right\}\right|<2\gamma+\varepsilon.
		\end{equation}
		To see this, it is sufficient to add some $0\leq l<M$ to $N$ in the definition of $Q_{x_0}(\gamma^2,K,\{\beta_M\}_{M=1}^\infty)$ for $\varepsilon/2$.
		
		Let $u\in Q_{x_0}(\gamma^2,K,\{\beta_M\}_{M=1}^\infty)$. Then there exists $M\geq M_0$, $N\geq N_{\varepsilon,M}$ and $0\leq n<\gamma MN$ such that the inequality (\ref{condition_of_escape_of_mass}) holds and the inequality (\ref{condition_of_entropy}) holds for some $MN$-choice $(\beta_{M,i_m})_{m=0,\dots,MN-1}$ for $a_2^nux_0$ with respect to $a_1$ and $\beta_M$.
		For this $(\beta_{M,i_m})_{m=0,\dots,MN-1}$ and each $l=0,1,\dots, M-1$, we write
		$$
		q\left((\beta_{M,i_m})_{m=0,\dots,MN-1},l\right)=\dist(i_l,i_{M+l},\dots,i_{(N-1)M+l})\in D_{k_M}.
		$$
		Then we have the convex combination
		$$
		q\left((\beta_{M,i_m})_{m=0,\dots,MN-1}\right)=\frac{1}{M}\sum_{l=0}^{M-1}q\left((\beta_{M,i_m})_{m=0,\dots,MN-1},l\right)
		$$
		of elements in $D_{k_M}$. Since $(\beta_{M,i_m})_{m=0,\dots,MN-1}$ satisfies the inequality (\ref{condition_of_entropy}) and $H(q)\ (q\in D_{k_M})$ is concave, it follows that there exists $l\in\{0,\dots,M-1\}$ such that
		\begin{equation}\label{condition_of_l_entropy}
			\frac{1}{M}H\left(q\left((\beta_{M,i_m})_{m=0,\dots,MN-1},l\right)\right)<2\gamma+\varepsilon.
		\end{equation}
		For $M\geq M_0$, $N\geq N_{\varepsilon,M}$, $0\leq n<\gamma MN$ and $0\leq l<M$, we write $S(M,N,n,l)$ for the set of $u\in\overline{B^U_1}$ such that $u$ satisfies the inequality (\ref{condition_of_escape_of_mass}) for the above $M$, $N$ and $n$, and there exists an $MN$-choice $(\beta_{M,i_m})_{m=0,\dots,MN-1}$
		for $a_2^nux_0$ with respect to $a_1$ and $\beta_M$ such that $q\left((\beta_{M,i_m})_{m=0,\dots,MN-1},l\right)$
		satisfies the inequality (\ref{condition_of_l_entropy}) for the above $l$. Then, by the above argument, we have
		\begin{equation}\label{coverofQ}
			Q_{x_0}(\gamma^2,K,\{\beta_M\}_{M=1}^\infty)\subset\bigcup_{M\geq M_0}\bigcup_{N\geq N_{\varepsilon,M}}\bigcup_{0\leq n<\gamma MN}\bigcup_{0\leq l<M}S(M,N,n,l).
		\end{equation}
		
		We fix $M\geq M_0$, $N\geq N_{\varepsilon,M}$, $0\leq n<\gamma MN$ and $0\leq l<M$ and construct a good finite cover of $S(M,N,n,l)$. For each $u\in S(M,N,m,l)$, there exists an $MN$-choice for $a_2^nux_0$ with respect to $a_1$ and $\beta_M$ satisfying the above condition. We take and fix one of such $MN$-choices and write $\left(\beta_{M,i_m(u)}\right)_{m=0,\dots,MN-1}$. From the inequality (\ref{condition_of_l_entropy}), we have
		$$
		H\left(q((\beta_{M,i_m(u)})_{m=0,\dots,MN-1},l)\right)=H\left(\dist(i_l(u),i_{M+l}(u),\dots,i_{(N-1)M+l}(u))\right)<M(2\gamma+\varepsilon)
		$$
		and hence
		$
		(i_l(u),i_{M+l}(u),\dots,i_{(N-1)M+l}(u))\in R(k_M,N,M(2\gamma+\varepsilon)).
		$
		We define the subset $E(u)$ of $U$ by
		\begin{align*}
			E(u)=&\left\{v\in a_2^n\overline{B^U_1}a_2^{-n}\left|\ a_1^jva_2^nx_0\in\beta_{M,i_j(u)},\right.\right.\\
			&\qquad\qquad\qquad\qquad0\leq j<l\ {\rm and}\ a_1^{rM+l}va_2^nx_0\in\beta_{M,i_{rM+l}(u)},\ 0\leq r<N\Bigl\}.
		\end{align*}
		It follows that $a_2^nua_2^{-n}\in E(u)$. We write $\mathscr{E}(M,N,n,l)=\left\{E(u)\left|u\in S(M,N,n,l)\right.\right\}$. Then we have
		\begin{equation}\label{coverS}
			S(M,N,n,l)\subset\bigcup_{E\in\mathscr{E}(M,N,n,l)}a_2^{-n}Ea_2^n.
		\end{equation}
		Since, for each $u\in S(M,N,n,l)$, $E(u)$ is determined by $i_0(u),\dots,i_{l-1}(u)\in\{1,\dots,k_M\}$ and $(i_l(u),i_{M+l}(u),\dots,i_{(N-1)M+l}(u))\in R(k_M,N,M(2\gamma+\varepsilon))$ and $N\geq N_{\varepsilon,M}$, we have from the inequality (\ref{Mt_countchoice}) that
		\begin{equation}\label{estimatenumber}
			\left|\mathscr{E}(M,N,n,l)\right|\leq (k_M)^l\left|R(k_M,N,M(2\gamma+\varepsilon))\right|\leq (k_M)^le^{MN(2\gamma+2\varepsilon)}.
		\end{equation}
		
		We take each $E\in\mathscr{E}(M,N,n,l)$. Using Lemma \ref{partition_B}, we partition $E$ into small subsets such that their diameters are about $e^{-MN}$ and find the number of subsets which appear. We take $u\in S(M,N,n,l)$ such that $E=E(u)$. By the definition of $E=E(u)$, it follows that
		$$
		va_2^nx_0\in\bigcap_{j=0}^{l-1}a_1^{-j}\beta_{M,i_j(u)}\cap\bigcap_{r=0}^{N-1}a_1^{-rM-l}\beta_{M,i_{rM+l}(u)}
		$$
		for all $v\in E$. From this and (\ref{patitionalphacoverbeta}), it can be seen that there exists $j_m(u)\in\{1,\dots,k\}$ for $m=0,\dots,MN-1$ such that
		\begin{equation}\label{coverbyB}
			va_2^nx_0\in\bigcap_{m=0}^{MN-1}a_1^{-m}B_{j_m(u)}
		\end{equation}
		for all $v\in E$.
		Here, $a_2^nua_2^{-n}\in E=E(u)$. In addition,
		since $u\in S(M,N,n,l)$, $u$ satisfies the inequality (\ref{condition_of_escape_of_mass}) and then
		$
		(MN)^{-1}\left|\left\{m\in\{0,\dots,MN-1\}\left|\ a_1^ma_2^nux_0\in X\setminus K\right.\right\}\right|<2\gamma+\varepsilon.
		$
		Hence, it follows that
		\begin{equation}\label{small_escape_of_mass}
			\frac{1}{MN}\left|\left\{m\in\{0,\dots,MN-1\}\left|\ j_m(u)=k\right.\right\}\right|<2\gamma+\varepsilon.
		\end{equation}
		Since
		$E=E(u)\subset a_2^n\overline{B^U_1}a_2^{-n}=\left\{v\in U\left||v_1|\leq e^n,|v_2|\leq e^{2n}\right.\right\}$, we can partition $E$ into at most $\lceil\rho^{-1}2e^n\rceil\lceil\rho^{-1}2e^{2n}\rceil\leq 9\rho^{-2}e^{3n}$ subsets so that each $E'$ of the subsets satisfies $\diam E'\leq\rho$.
		From (\ref{coverbyB}), we can apply Lemma \ref{partition_B} to each $E'$ and use the inequality (\ref{small_escape_of_mass}). Then we see that $E$ is partitioned into at most
		$9\rho^{-2}e^{3n}\cdot(4e^3)^{MN(2\gamma+\varepsilon)}$ subsets so that each $E'$ of the subsets satisfies
		$\diam_1 E'\leq\rho e^{-2(MN-1)}$ and $\diam_2 E'\leq\rho e^{-(MN-1)}$.
		
		Let $M\geq M_0$, $N\geq N_{\varepsilon,M}$, $0\leq n<\gamma MN$ and $0\leq l<M$.
		For each $E\in\mathscr{E}(M,N,n,l)$, we take a decomposition of $E$: 
		\begin{equation*}
			E=\bigsqcup_{1\leq b\leq 9\rho^{-2}e^{3n}\cdot(4e^3)^{MN(2\gamma+\varepsilon)}}E_b
		\end{equation*}
		such that $\diam E_b\leq \rho e^{-(MN-1)}$. Then, by (\ref{coverS}) we have
		\begin{equation}\label{smallcover_ofS}
			S(M,N,n,l)\subset\bigcup_{E\in \mathscr{E}(M,N,n,l)}\bigcup_{1\leq b\leq 9\rho^{-2}e^{3n}\cdot(4e^3)^{MN(2\gamma+\varepsilon)}}a_2^{-n}E_ba_2^n.
		\end{equation}
		We write $\widetilde{E_b}=a_2^{-n}E_ba_2^n$. Then, since $U$ is stable for the conjugation with $a_2^{-n}$, it follows that
		\begin{equation}\label{diamEtilde}
			\diam\widetilde{E_b}\leq\diam E_b\leq \rho e^{-(MN-1)}.
		\end{equation}
		By (\ref{coverofQ}) and (\ref{smallcover_ofS}), if we define
		\begin{align*}
			\widetilde{\mathscr{E}}_{\varepsilon,M_0}=&\left\{\left.\widetilde{E_b}\ \right|M\geq M_0, N\geq N_{\varepsilon,M},0\leq n< \gamma MN,0\leq l<M,\right.\\
			&\quad\quad\quad E\in\mathscr{E}(M,N,n,l),1\leq b\leq 9\rho^{-2}e^{3n}\cdot(4e^3)^{MN(2\gamma+\varepsilon)}\Bigr\},
		\end{align*}
		then $\widetilde{\mathscr{E}}_{\varepsilon,M_0}$ is a countable cover of $Q_{x_0}(\gamma^2,K,\{\beta_M\}_{M=1}^\infty)$. If we fix
		$M\geq M_0$, $N\geq N_{\varepsilon,M}$, $0\leq n<\gamma MN$ and $0\leq l<M$, then, by the inequality (\ref{estimatenumber}), the number of elements of $\widetilde{\mathscr{E}}_{\varepsilon,M_0}$ corresponding to $M$, $N$, $n$ and $l$ is bounded by
		\begin{equation*}
			\left|\mathscr{E}(M,N,n,l)\right|\cdot 9\rho^{-2}e^{3n}\cdot (4e^3)^{MN(2\gamma+\varepsilon)}<
			9\rho^{-2}(k_M)^le^{3n}(4e^4)^{MN(2\gamma+2\varepsilon)}.
		\end{equation*}
		We put $\lambda_\varepsilon=3\gamma+(4+\log4)(2\gamma+2\varepsilon)+\varepsilon$. Then, by the inequality (\ref{diamEtilde}) and (\ref{solargeN}), we have
		\begin{align*}
			&\sum_{\widetilde{E}\in\widetilde{\mathscr{E}}_{\varepsilon,M_0}}\left(\diam\widetilde{E}\right)^{\lambda_\varepsilon}\\
			\leq&\sum_{M\geq M_0}\sum_{N\geq N_{\varepsilon,M}}\sum_{0\leq n<\gamma MN}\sum_{0\leq l<M}
			9\rho^{-2}(k_M)^le^{3n}(4e^4)^{MN(2\gamma+2\varepsilon)}(\rho e^{-(MN-1)})^{\lambda_\varepsilon}\\
			\leq&\ 9e^{\lambda_\varepsilon}\rho^{\lambda_\varepsilon-2}\sum_{M\geq M_0}M(k_M)^M\sum_{N\geq N_{\varepsilon,M}}\sum_{0\leq n<\gamma MN}e^{3n}(4e^4)^{MN(2\gamma+2\varepsilon)}e^{-MN\lambda_\varepsilon}\\
			\leq&\ 9e^{\lambda_\varepsilon}\rho^{\lambda_\varepsilon-2}\sum_{M\geq M_0}M(k_M)^M\sum_{N\geq N_{\varepsilon,M}}2\gamma MNe^{3\gamma MN}(4e^4)^{MN(2\gamma+2\varepsilon)}e^{-MN(3\gamma+(4+\log 4)(2\gamma+2\varepsilon))}e^{-MN\varepsilon}\\
			\leq&\ 18\gamma e^{\lambda_\varepsilon}\rho^{\lambda_\varepsilon-2}\sum_{M\geq M_0}M^2(k_M)^M\sum_{N\geq N_{\varepsilon, M}}Ne^{-MN\varepsilon}\\
			\leq&\ 18\gamma e^{\lambda_\varepsilon}\rho^{\lambda_\varepsilon-2}\sum_{M\geq M_0}2^{-M}.
		\end{align*}
		From the inequality (\ref{diamEtilde}), we have $\diam \widetilde{\mathscr{E}}_{\varepsilon,M_0}\leq \rho e^{-M_0}\to 0$ as $M_0\to\infty$. Furthermore, the above right hand side converges to $0$ as $M_0\to\infty$. These imply that $\dim_H Q_{x_0}(\gamma^2,K,\{\beta_M\}_{M=1}^\infty)$ $\leq \lambda_\varepsilon$. Since $0<\varepsilon<1$ is arbitrary and $\lambda_\varepsilon\to 3\gamma+2(4+\log 4)\gamma<15\gamma$ as $\varepsilon\to 0$, we obtain $\dim_H Q_{x_0}(\gamma^2,K,\{\beta_M\}_{M=1}^\infty)\leq 15\gamma$ and complete the proof.
	\end{proof}


\begin{thebibliography}{99}
		\bibitem[Bow73]{Bow73} R. Bowen, Topological entropy for noncompact sets, {\it Trans. Amer. Math. Soc.} {\bf184} (1973), 125-136.
		\bibitem[CSD55]{CSD55} J.W.S. Cassels and H.P.F. Swinnerton-Dyer, On the product of three homogeneous linear forms and indefinite ternary quadratic forms, {\it Philos. Trans. Roy. Soc. London Ser. A} {\bf 248} (1955), 73-96.
		\bibitem[EK03]{EK03} M. Einsiedler and A. Katok, Invariant measures on $G/\Gamma$ for split simple Lie groups $G$, {\it Comm. Pure Appl. Math.} {\bf 56} (2003), no. 8, 1184-1221.
		\bibitem[EKL06]{EKL06} M. Einsiedler, A. Katok and E. Lindenstrauss, Invariant measures and the set of exceptions to Littlewood's conjecture, {\it Ann. of Math.} {\bf 164} (2006), 513-560.
		\bibitem[EW11]{EW11} M. Einsiedler and T. Ward, {\it Ergodic Theory with a view towards Number Theory, Graduate Text in Mathematics {\bf 259}}, Springer-Verlag, London, 2011.
		\bibitem[Lin10]{Lin10} E. Lindenstrauss, Equidistribution in homogeneous spaces and number theory. In: Proceedings of the International Congress of Mathematicians. Volume I, 531-557, Hindustan Book Agency, New Delhi, 2010.
		\bibitem[MT94]{MT94} G.A. Margulis and G.M. Tomanov, Invariant measures for actions of unipotent groups over local fields on homogeneous spaces, {\it Invent. Math.} {\bf 116} (1994), 347-392.
		\bibitem[Phe01]{Phe01} R.R. Phelps, {\it Lectures on Choquet's Theorem, Lecture Notes in Mathematics {\bf 1757}}, Springer-Verlag, Berlin Heidelberg, 2001.
		\bibitem[PV00]{PV00} A.D. Pollington and S. Velani, On a problem in simultaneous Diophantine approximation: Littlewood's conjecture, {\it Acta Math.} {\bf 185} (2000), no. 2, 287-306.
		\bibitem[PVZZ22]{PVZZ22} A.D. Pollington, S. Velani, A. Zafeiropoulos and E. Zorin, Inhomogeneous Diophantine approximation on $M_0$-sets with restricted denominators, {\it Int. Math. Res. Not. IMRN} (2022), no. 11, 8571-8643.
		\bibitem[Usu22]{Usu22} S. Usuki, $\times a$ and $\times b$ empirical measures, the irregular set and entropy, arXiv: 2205.06605.
		\bibitem[Wal82]{Wal82} P. Walters, {\it An Introduction to Ergodic Theory, Graduate Text in Mathematics {\bf 79}}, Springer-Verlag, New York, 1982.
	\end{thebibliography}
\end{document}